\documentclass[11pt]{amsart}

\usepackage{amssymb}
\usepackage[varbb, varg, smallerops]{newtxmath}
\usepackage[scaled=0.87]{inconsolata} 
\usepackage[T1]{fontenc}
\usepackage{amsfonts}
\usepackage{comment}
\usepackage[a4paper,margin=1.1in]{geometry} 
\usepackage{tikz}
\usepackage{tikz-cd}
\usetikzlibrary{arrows}
\tikzcdset{arrow style=tikz,
diagrams={>={Stealth[round,length=4pt,width=6pt,inset=3pt]}}}
\usepackage{graphicx}
\usepackage[all]{xy}
\makeatletter
\newcommand*{\tightdisplaymath}{\abovedisplayskip\z@\belowdisplayskip\z@}

\usepackage{calligra}

\usetikzlibrary{backgrounds,calc,shadings,shapes.arrows,shapes.symbols,shadows}
\usepackage{amsmath, amsfonts, amsthm, amscd}	

\newcommand{\vers}{Some properties of the $\mbox{A}_{\infty}$-nerve}
\title[\vers]{Some properties of the $\mbox{A}_{\infty}$-nerve}

\usepackage{hyperref}
\hypersetup{colorlinks=false}

\author{Mattia Ornaghi}
\address{\parbox{0.9\textwidth}{Universit\`a degli Studi di Milano\\
Dipartimento di Matematica\\
Via Cesare Saldini 50, 20133 Milano, Italy}}
\email{mattia12.ornaghi@gmail.com}

\theoremstyle{definition}
\newtheorem{defn}{Definition}[section]

\newtheorem{thm}{Theorem}[section]

\newtheorem{lem}[thm]{Lemma}

\newtheorem{prp}[thm]{Proposition}
\newtheorem{cor}[thm]{Corollary}

\newtheorem*{namedthm}{Theorem}

\theoremstyle{remark}
\newtheorem{rem}{Remark}[section]
\newtheorem{exmp}{Example}[section]

\newcommand{\Ain}{\mbox{A$_{\infty}$}}

\newcommand{\aCat}{\mbox{A$_{\infty}$Cat}}

\newcommand{\DgCat}{\mbox{DgCat}}

\newcommand{\A}{\mathscr{A}}

\newcommand{\B}{\mathscr{B}}

\newcommand{\F}{\mathscr{F}}

\thanks{The author was supported by the research project FARE 2018 HighCaSt (grant number R18YA3ESPJ) and by ERC Advanced
Grant-101095900-TriCatApp}

\begin{document}

\date{\today}

\maketitle

\begin{abstract}
The aim of this paper is to prove that the A$_{\infty}$-nerve of two quasi-equivalent A$_{\infty}$-categories (linear over a commutative ring) are weak-equivalent in the Joyal model structure.\ As a consequence we prove that the A$_{\infty}$-nerve of a pretriangulated A$_{\infty}$-category is a stable $\infty$-category.
\end{abstract}

\setcounter{tocdepth}{1}

\section{Introduction}

The notion of \emph{triangulated category} was developed in the 60s by Jean-Louis Verdier, under the guidance of Alexandre Grothendieck,
in order to capture the additional structure on the derived category of an abelian category.\
Nowadays triangulated categories played an important role in algebraic geometry, even though they have some drawbacks, for example the non-functoriality of the mapping cone or the non-existence of homotopy colimits and homotopy limits.\\

For this purpose, in the 90s it was developed the notion of \emph{pretriangulated envelope} of a \emph{differential graded category} (from now on \emph{dg-category}), and of a $\mbox{A}_\infty$-category.\ Roughly speaking, pretriangulated dg-categories (resp.\ $\mbox{A}_\infty$-categories) are dg-categories (resp.\ $\mbox{A}_\infty$-categories) whose homotopy category is "canonically" triangulated.\ It means that pretriangulated dg and $\mbox{A}_\infty$-categories can be seen as "enhanced" triangulated categories.\ A more recent way to enhance a triangulated category is via a stable $\infty$-category.\ 
More precisely, a stable $\infty$.category is a pointed $\infty$-category that is complete and closed under loop spaces, whose homotopy category is triangulated.\\

It is a folklore belief (see for example \cite[\S 2.1.1]{BFN}\footnote{note that in this paper $\mathbb{K}$ is a field of characteristic zero}) that, over a commutative ring $\mathbb{K}$, the notions of ($\mathbb{K}$-linear) pretriangulated $\mbox{A}_{\infty}$-categories, pretriangulated dg-categories and stable $\infty$-categories are equivalent, under suitable localization.\ Unfortunately we cannot find any satisfying reference in the existing literature.\\ 

Regarding the category of pretriangulated dg-categories linear over a commutative ring $\mathbb{K}$, in 2013 Lee Cohn (cf. \cite{Coh}) proves that the nerve of the category of dg-categories localized on Morita equivalences is $\infty$-equivalent to the $\infty$-category of stable idempotent complete $\infty$-categories enriched over the Eilenberg-MacLane spectra $H\mathbb{K}$ (see also \cite{Doni}). 
This fact proves that the categorical nerve of the category of dg-categories (localizing on Morita equivalence) is equivalent to an idempotent complete stable $\infty$-category.
The problem is that the strategy used by Cohn does not extend to pretriangulated $\mbox{A}_{\infty}$-categories. However in 2015 Giovanni Faonte proved that, when $\mathbb{K}$ is a field, the dg-nerve of a pretriangulated dg-category (in the sense of \cite{BoKa}) is a stable $\infty$-category.\\

The first aim of the present work is to extend the same result to pretriangulated $\mbox{A}_{\infty}$-categories (resp. dg-categories), linear over a commutative ring $\mathbb{K}$.\ 
The second goal is to investigate some new possibilities offered by the $\mbox{A}_{\infty}$-nerve defined by Giovanni Faonte and Jacob Lurie.\ In particular, we clarify the relationship between the dg-nerve and the A$_\infty$-nerve.\ We will prove:

\begin{namedthm}[\textbf{\ref{M}}]
The $\mbox{A}_{\infty}$-nerve sends quasi-equivalences of (strictly unital) $\mbox{A}_{\infty}$-categories in weak-equivalences of $\infty$-categories. 
\end{namedthm}
 
\begin{namedthm}[\textbf{\ref{Ttr}}]
Let $\mathscr{A}$ be a pretriangulated $\mbox{A}_{\infty}$-category then $\mbox{N}_{\tiny{\mbox{A}}_{\infty}}(\mathscr{A})$ is a stable $\infty$-category.\ The functor induced in the homotopy categories is an equivalence of triangulated categories.\ Moreover, $\mathscr{A}$ is idempotent complete if and only if $\mbox{N}_{\tiny{\mbox{A}}_{\infty}}(\mathscr{A})$ is an idempotent complete stable $\infty$-category.
\end{namedthm}
This means that the $\mbox{A}_{\infty}$-nerve of a (strictly unital) pretriangulated $\mbox{A}_{\infty}$-category is a stable $\infty$-category, and the nerve induces a triangulated functor at homotopy categories level.
Unfortunately, using the $\mbox{A}_{\infty}$-nerve, we do not have an equivalence of $\infty$-categories between the nerve of the category of the $\mbox{A}_{\infty}$-categories and a stable $\infty$-category as in the case of the category of dg-categories (localized over Morita equivalences).\
On the other hand, if $\mathbb{K}$ is a field, we recently prove (see \cite{CO}) that the category of A$_{\infty}$-categories linear over $\mathbb{K}$, is a fibrant category.\ As in the case of dg-categories, the A$_{\infty}$-nerve preserves the fibrations (see Theorem \ref{ganzo}).

\subsection*{Acknowledgements} 
The author wants to thank Paolo Stellari for proposing the topic and Gon\c calo Tabuada for the valuable advice.
The author is also very grateful to Francesco Genovese, Marco Manetti and Zhao Yan for many useful and interesting discussions and to Emily Riehl and Bruno Vallette for the thorough explanations on their results. I am grateful to the anonymous referee for a very careful reading of the manuscript.

\tableofcontents

\section{$\mbox{A}_{\infty}$-modules, quasi-equivalences and pretriangulated $\mbox{A}_{\infty}$-categories}

In this section we will recall some basic definitions about the $\mbox{A}_{\infty}$-categories and we will discuss the pretriangulated envelopement of the $\mbox{A}_{\infty}$-categories.

\subsection{Brief background on $\mbox{A}_{\infty}$-categories} 
First of all we give some information about $\mbox{A}_{\infty}$-categories well known to the experts, a good reference for the theory of the $\mbox{A}_{\infty}$-categories is \cite{Sei}. We omit the notion of dg-category that will further be useful, cf. \cite{Kel} for a survey about this topic.\\
\\
Let $\mathbb{K}$ be a commutative ring.

\begin{defn}[$\mbox{A}_\infty$-category]
We define an \emph{${\mbox{A}}_\infty$-category} $\mathscr{A}$ to be a set of objects, a $\mathbb{K}$-linear graded module $\mbox{Hom}_{\mathscr{A}}(x_0,x_1)$ for any pair of objects, and $\mathbb{K}$-linear maps 
$$m^d_{\mathscr{A}}:\mbox{Hom}_{\mathscr{A}}(x_{d-1},x_d)\otimes...\otimes\mbox{Hom}_{\mathscr{A}}(x_{0},x_1)\to\mbox{Hom}_{\mathscr{A}}(x_{0},x_d)[2-d]\footnote{where $\mbox{[$n$]}$ denotes the shift of a graded $\mathbb{K}$-module down by an integer $n$.},$$
for every $d>0$. Moreover the maps above must verifying the followings:
\begin{align}
\sum_{m,n} (-1)^{\dagger_n}m^{d-m+1}_{\mathscr{A}}(a_d,...,a_{n+m+1},m^m_{\mathscr{A}}(a_{n+m},...,a_{n+1}),a_n,...,a_1) &=0.
\end{align}
where $1\le m\le d$, $0\le n\le d-m$ and $\dagger_n=\mbox{deg$(a_1)$}+...+\mbox{deg$(a_n)$}-n$.
\end{defn}

\begin{exmp}
Every differential graded category $\mathscr{C}$ is an A$_{\infty}$-category such that $m_{\mathscr{C}}^{n\ge3}=0$.
\end{exmp}

\begin{defn}[Unit]\label{unil}
Given an object $x$ in $\mathscr{A}$. We define the \emph{unit} of $x$ to be a morphism of degree zero, denoted by $1_x$, such that: 
\begin{itemize}
\item[(u1)] $m^2_{\mathscr{A}}(f,1_x)=f$ and $m^2_{\mathscr{A}}(1_x,f)=(-1)^{\tiny\mbox{deg $f$}}f$, for every morphism $f$;
\item[(u2)] $m^n_{\mathscr{A}}(...,1_x,...)=0$, for all $n>2$.
\end{itemize}
\end{defn}

\begin{defn}[Strictly unital $\mbox{A}_{\infty}$-functor]
We define an \emph{$\mbox{A}_{\infty}$-functor} $\mathscr{F}:\mathscr{A}\to\mathscr{A}'$ to be a map $\mathscr{F}_0$ between the objects of $\mathscr{A}$ and $\mathscr{A}'$ and a collection of $\mathbb{K}$-linear maps (for all $n\ge1$): $$\mathscr{F}_n:\mbox{Hom}_\mathscr{C}(x_{n-1},x_{n})\otimes...\otimes \mbox{Hom}_\mathscr{C}(x_{0},x_{1})\to \mbox{Hom}_{\mathscr{D}}(\mathscr{F}_0(x_0),\mathscr{F}_0(x_n))[1-n]$$ such that for every $0<m\le n$:
\begin{align*}
\displaystyle\sum_{r\ge1} &\displaystyle\sum_{s_1+...+s_r=n}m^{r}_{\B}\big(\F^{s_r}(f_n,...,f_{n-s_r+1}),...,\F^{s_1}(f_{s_1},...,f_1)\big)=\\
&=\displaystyle\sum_{m=1}^n\displaystyle\sum_{k=0}^{n-m}(-1)^{\dagger_k} \F^{n-m+1}\big(f_n,...,f_{k+m+1},m^{m}_{\A}(f_{k+m},...,f_{k+1}),f_k,...,f_1\big).
\end{align*}
Where $\dagger_k=\mbox{deg}(f_d)+...+\mbox{deg}(f_1)-d$.\\
Moreover we require that the unit have to be preserved by $\mathscr{F}^1$ and $\mathscr{F}^n(...,1_x,...)=0$ for all $n\ge2$.
\end{defn}

From now on we consider only \emph{strictly unital $\mbox{A}_{\infty}$-categories} i.e. $\mbox{A}_{\infty}$-categories with units, according to the definition \ref{unil}, and strictly unital A$_{\infty}$-functors.\ The reason why we work with strictly unital A$_{\infty}$-categories will become clear in section \ref{nervez} when we define the A$_{\infty}$-nerve.

\subsection{Quasi-equivalences between $\mbox{A}_{\infty}$-categories}
In this subsection we define the \emph{homotopy category} of an A$_{\infty}$-category and the notion of \emph{quasi-equivalence}.

\begin{defn}[Homotopy category]
Let $\mathscr{A}$ be an $\mbox{A}_{\infty}$-category, we define the \emph{homotopy category} $\mbox{Ho}(\mathscr{A})$ of $\mathscr{A}$, as the category whose objects are the objects of $\mathscr{A}$ and whose morphisms, for $x$ and $y\in\mbox{Obj}(\mathscr{A})$, are given by the quotients$$\mbox{Hom}_{\tiny\mbox{Ho}(\mathscr{A})}(x,y):=\frac{Z^0(\mbox{Hom}_{\mathscr{A}}(x,y))}{B^0(\mbox{Hom}_{\mathscr{A}}(x,y))}=H^0(\mbox{Hom}_{\mathscr{A}}(x,y)),$$
where $Z^0(\mbox{Hom}_{\mathscr{A}}(x,y)):=\mbox{Ker}(m^1_{\mathscr{A}}:\mbox{Hom}^0_{\tiny\mathscr{A}}(x,y)\to\mbox{Hom}^{1}_{\tiny\mathscr{A}}(x,y))$ and\\ $B^0(\mbox{Hom}_{\mathscr{A}}(x,y)):=\mbox{Im}(m^1_{\mathscr{A}}:\mbox{Hom}^{-1}_{\tiny\mathscr{A}}(x,y)\to\mbox{Hom}^{0}_{\tiny\mathscr{A}}(x,y)).$
\end{defn}

\begin{defn}[Quasi-equivalence]
Let $\mathscr{A}$, $\mathscr{A}'$ be $\mbox{A}_{\infty}$-categories, we say that an $\mbox{A}_{\infty}$-functor $\mathcal{f}\mathscr{F}^n\mathcal{g}:\mathscr{A}\to\mathscr{A}'$ is a \emph{quasi-equivalence} if:\begin{itemize}
\item[(we1)] $\mbox{Ho$(\mathscr{F})$}:\mbox{Ho}(\mathscr{A})\to \mbox{Ho}(\mathscr{A}')$ is an equivalence of categories.
\item[(we2)] $\mathscr{F}^1:\mbox{Hom}^{\cdot}_{\mathscr{A}}(x,y)\to\mbox{Hom}^{\cdot}_{\mathscr{A}'}\big(\mathscr{F}^0(x),\mathscr{F}^0(y)\big)$ is a  quasi-isomorphism.
\end{itemize}
\end{defn}

\begin{exmp}
Two dg-categories which are quasi-equivalent are quasi-equivalent as $\mbox{A}_{\infty}$-categories.\ Note that two A$_{\infty}$ (resp. dg) -categories $\mathscr{A}$ and $\mathscr{B}$ are quasi-equivalent if there exists a zig-zag of quasi-equivalences whose source is $\mathscr{A}$ and target $\mathscr{B}$. 
\end{exmp}

We have this following fundamental result:

\begin{thm}[\cite{COS}, \cite{COS2}]\label{}
We have a functor $U:\aCat\to\DgCat$ providing a DK adjunction of categories
\begin{align*}
\xymatrix{
U:\aCat\ar@<-0.5ex>[r]&\ar@<-0.5ex>[l]\DgCat:i
}
\end{align*}
Where $\aCat$ denotes the category of strictly unital A$_{\infty}$-categories and $\DgCat$ denotes the category of dg-categories.\
In particular, given $\mathscr{A}\in\aCat$, we have a quasi-equivalence $\A\to\mbox{U}(\A)$.
\end{thm}

\subsection{Pretriangulated $\mbox{A}_{\infty}$-categories} The next definition is due to Kontsevich, we refer to \cite{BLM} for the proofs.\
 Let $\mathscr{A}$ be a $\mathbb{K}$-linear $\mbox{A}_{\infty}$-category.

\begin{defn}[Shift category and shift functor]
We define the category $\Sigma(\mathscr{A})$ to be the $\mbox{A}_{\infty}$-category such that $\mbox{Obj}(\Sigma\mathscr{A})=(\mbox{Obj}(\mathscr{A}))\times\mathbb{Z}$, and morphisms are defined as follow $$\mbox{Hom}_{\Sigma(\mathscr{A})}(x[n],y[m]):=\mbox{Hom}_{\mathscr{A}}(x,y)[m-n],$$where $x$, $y\in\mathscr{A}$. The endofunctor sending $x[n]$ to $x[n+1]$ is called \emph{shift functor}.
\end{defn}

\begin{defn}[Closed under shift]
We say that $\mathscr{A}$ is \emph{closed under shift} if $\mathscr{A}\hookrightarrow\Sigma(\mathscr{A})$ is a quasi-equivalence.
\end{defn}

As in the case of dg-categories the set of $\mbox{A}_\infty$-twisted complexes has an $\mbox{A}_{\infty}$-structure (see \cite{Sei}, \cite{BLM}) whose homotopy category is triangulated.\ We denote such an $\mbox{A}_\infty$-category by $\mbox{pretr($\mathscr{A}$)}$. Moreover we have an $\mbox{A}_{\infty}$-functor 
\begin{align*}
i_{\tiny\mbox{A}_{\infty}}:\mathscr{A}\hookrightarrow \mbox{pretr}_{\tiny\mbox{A}_{\infty}}(\mathscr{A})
\end{align*} 
and it was proven that the construction is functorial (cf. $\mathcal{x}$3 of \cite{Sei}). Given an $\mbox{A}_{\infty}$-morphism $\mathscr{F}$ we denote by $\mbox{pretr}_{\tiny\mbox{A}_{\infty}}\mathscr{F}$ the induced functor.

\begin{defn}[Pretriangulated $\mbox{A}_{\infty}$-categories]
We say that an $\mbox{A}_{\infty}$-category $\mathscr{A}$ is \emph{pretriangulated} if $\mathscr{A}$ is closed under shift and the functor $i_{\tiny\mbox{A}_{\infty}}:\mathscr{A}\hookrightarrow \mbox{pretr}_{\tiny\mbox{A}_{\infty}}(\mathscr{A})$ is a quasi-equivalence.
\end{defn}

\begin{rem}
If $\mathscr{C}$ is a dg-category $\mbox{pretr}(\mathscr{C})=\mbox{pretr}_{\tiny\mbox{A}_{\infty}}(\mathscr{C})$, where $\mbox{pretr}(\mathscr{C})$ denotes the pretriangulated envelope of the dg-category $\mathscr{C}$ according to the notation of \cite{Kel}.
\end{rem}

We have the following \cite[Lemma 3.25]{Sei}:

\begin{thm}{}
Let $\mathscr{F}:\mathscr{A}\to\mathscr{B}$ be a quasi-equivalence between two $\mbox{A}_{\infty}$-categories then 
$$\mbox{pretr}_{\tiny\mbox{A}_{\infty}}\mathscr{F}:\mbox{pretr}_{\tiny\mbox{A}_{\infty}}(\mathscr{A})\to\mbox{pretr}_{\tiny\mbox{A}_{\infty}}(\mathscr{B})$$ 
is a quasi-equivalence.
\end{thm}

By the following diagram we deduce that, an $\mbox{A}_{\infty}$-category $\mathscr{A}$ is $\mbox{A}_{\infty}$-pretriangulated, if and only if $
{\mbox{U}}(\mathscr{A})$ is pretriangulated (as dg-category).
\[
\xymatrixrowsep{0.45in}
\xymatrix{
\mathscr{A}\ar[rr]^{\sim}\ar@{^{(}->}[d]&& {\mbox{U}}(\mathscr{A})\ar@{^{(}->}[d]\\
\mbox{pretr}_{\tiny\mbox{A}_{\infty}}(\mathscr{A})\ar[rr]^{\sim}&& \mbox{pretr}({\mbox{U}}(\mathscr{A}))
}
\] 
\begin{defn}[Idempotent complete]
We say that an additive category $\mathscr{K}$ is \emph{idempotent complete} if any endomorphism $E:k\to k$ such that $E^2=E$ (idempotent) is such that $k=\mbox{Im}(E)\oplus\mbox{ker}(E)$.
\end{defn}

According to \cite[Definition 1.2]{BaSc}, in general, we can always embed an additive category in a idempotent complete category (we denote by $(-)^{ic}$ such an embedding) moreover if $\mathscr{K}$ is a triangulated category we have the following \cite[Theorem 1.5]{BaSc}:

\begin{prp}
If $\mathscr{K}$ is a triangulated category, its idempotent completion $(\mathscr{K})^{ic}$ admits a unique triangulated structure such that the canonical functor $(-)^{ic}$ is exact.
\end{prp}

\begin{defn}[Idempotent complete]
We say that a pretriangulated $\mbox{A}_{\infty}$-category $\mathscr{T}$ (resp. dg-category) is \emph{idempotent complete} if the homotopy category $\mbox{Ho}(\mathscr{T})$ is idempotent complete.
\end{defn}

\section{$\mbox{A}_{\infty}$-nerve}

\subsection{Brief background on $\infty$-categories}

We briefly recall the basics about $\infty$-categories.\ 
The non-expert reader can have a look at Chapter 1 and 2 of \cite{Lur1}.

\begin{defn}[Minimal $\mathbb{K}$-linear category]
Let $n$ be a positive integer (or zero). We define the \emph{minimal $\mathbb{K}$-linear category} $\mbox{[$n$]}_{\mathbb{K}}$ to be the category whose objects are the positive integers $\mathcal{f}0,1,2,...,n\mathcal{g}$ and the morphisms are defined by 
$$\mbox{Hom}_{\tiny\mbox{[$n$]}_{\mathbb{K}}}(i,k)=
\begin{cases} 0_{\mathbb{K}}, & \mbox{if $i>k$} \\ \langle j_{ik}\rangle_{\mathbb{K}}, & \mbox{if $i<k$} \\ \langle1_{\tiny\mathbb{K}}\rangle_{\mathbb{K}}, & \mbox{if $i=k$}.
\end{cases}$$
where $0_{\mathbb{K}}$ is the zero module and $\langle j_{ik}\rangle_{\mathbb{K}}$ is the $\mathbb{K}$-module generated by the element $j_{ik}$.
The composition is defined as follow; let $i_1<i_2<i_3$ be positive integers. Then:
\begin{equation}\label{compos}
\cdot:=\mbox{Hom}_{\tiny\mbox{[$n$]}_{\mathbb{K}}}(i_2,i_3)\otimes_{\mathbb{K}}\mbox{Hom}_{\tiny\mbox{[$n$]}_{\mathbb{K}}}(i_1,i_2)\to\mbox{Hom}_{\tiny\mbox{[$n$]}_{\mathbb{K}}}(i_1,i_3)
\end{equation}
is such that
\begin{equation}\label{composiz}
j_{i_2i_3}\cdot j_{i_1i_2}= j_{i_1i_3},
\end{equation}
where $j_{i_1i_3}$ is the unique morphism in $\mbox{Hom}_{\tiny\mbox{[$n$]}_{\mathbb{K}}}(i_1,i_3)$.
\end{defn}

\begin{rem}
The definition above works even without the $\mathbb{K}$-linear enrichment.\ In this case, in equation (\ref{compos}) we take "$\times$", the usual categorical product of sets, instead of the tensor product. 
\end{rem}

\begin{defn}[Simplex category]
We define the \emph{simplex category} to be the category whose objects are the minimal $\mathbb{K}$-linear categories $\mbox{[$n$]}$ and whose morphisms are the functions $f$ such that $f(i)\le i$ and $f(i_1)\le f(i_2)$ if $i_1\le i_2$. We denote by $\Delta$ such a category.
\end{defn}

\begin{defn}[Simplicial set]
We define a \emph{simplicial set} to be a contravariant functor from the simplex category $\Delta$ to the category of sets. 
\end{defn}

We will denote by $\mbox{sSet}$ the category of simplicial sets.

\begin{exmp}
Given a positive integer $n$, the functor $\Delta^{n}$ defined as $\mbox{Hom}_{\Delta}(-,\mbox{[$n$]}):\Delta^{\tiny\mbox{op}}\to\mbox{Sets}$ is a simplicial set. Moreover for each $0\le i\le n$ the functor generated by all the maps $d^{j}:\mbox{[$n-1$]}\to\mbox{[$n$]}$ (which are the injective maps not having $j$ in the image), with $i\not=j$, is a subsimplicial set of $\Delta^n$. We call such a simplicial set \emph{$(n,i)$-horn} and we denote it by $\Lambda^n_i$.
\end{exmp}

\begin{defn}[$\infty$-category]
We define an \emph{$\infty$-category} to be a simplicial set $X$ such that, for every positive integer $n$ and every natural transformation $\phi:\Lambda^n_k\to X$, with $0<k<n$, there exists (at least) one map $\tilde{\phi}$ such that the following diagram:
\[
\xymatrix{
\Lambda^n_k\ar@{_(->}[d]\ar[r]^{\phi}&X\\
\Delta^n\ar@{-->}[ur]_{\tilde{\phi}}&
}
\]
commutes. 
\end{defn}

Let $X$ be an $\infty$-category, the \emph{objects} of $X$ are given by the elements of the set $X_0$. The \emph{(simplicial) set of morphisms} from $x$ to $y$, denoted by $\mbox{Map}_X(x,y)$, is given by the pullback of the following diagram:
\[
\xymatrix{
\mbox{Map}_X(x,y)\ar[d]\ar[r]&X_1\ar[d]^{(d,c)}\\
\ast\ar[r]_-{(x,y)}&X_0\times X_0
}
\]
where $d=X(d_1)\colon X_1\to X_0$, and $c=X(d_0)\colon X_1\to X_0$, see \cite[pp. 5]{riehl}

\begin{exmp}
Let $X$ be an $\infty$-category. Fixing two elements $x$ and $y\in X_0$, we get a simplicial set, denoted by $\mbox{Hom}^{\tiny\mbox{R}}_X(x,y)$, whose 0-simplices are 1-simplices in $X$ from $x$ to $y$, whose 1-simplices are
2-simplices of the form:
\[
\xymatrix{
&x\ar[dr]&\\
x\ar^{\tiny\mbox{1}_x}[ur]\ar[rr]&&y
}
\]
and whose $n$-simplices are $(n + 1)$-simplices whose target is $y$ and whose $(n + 1)^{\tiny\mbox{th}}$-face degenerates at $x$\footnote{cf. \textbf{Definition 3.5}}.
\end{exmp}

\begin{exmp}
Let $\mathscr{C}$ be a category, the simplicial set defined as the set of the compositions of $n$-arrows of $\mathscr{C}$, for every $n>0$, and as the set of objects of $\mathscr{C}$, if $n=0$, is an $\infty$-category. We call such a simplicial set the \emph{nerve} of $\mathscr{C}$ and we denote it by $\mbox{N}_{\tiny\mbox{Cat}}(\mathscr{C})$.
\end{exmp}

Given an $\infty$-category $X$, taking $f,g\in\mbox{Map}_X(x,y)$ we say that $f$ is \emph{homotopic} to $g$ if there exists a natural transformation $\sigma:\Delta^2\to X$ of the form:
\[
\xymatrix{
&x\ar[dr]^f&\\
x\ar^{\tiny\mbox{1}_x}[ur]\ar[rr]_g&&y
}
\]
We recall that the homotopy relation is an equivalence relation.

\begin{defn}[Homotopy category]
Let $X$ be an $\infty$-category.\ 
We define the \emph{homotopy category} $\mbox{Ho}(X)$ to be the (ordinary) category whose objects are the elements of $X_0$ and whose morphisms, fixed two objects $x$ and $y$, are given by the quotient of $\mbox{Map}_X(x,y)$ by the homotopy relation defined above.
\end{defn}

In other words the set of morphisms $\mbox{Hom}_{\tiny\mbox{Ho}(X)}(x,y)$ is given by $\pi_0(\mbox{Map}_X(x,y))$.

\subsection{Stable $\infty$-categories}

In this subsection we give the precise definitions of stable $\infty$-categories and exact functors between them.

\begin{defn}[Zero object in $\infty$-category]
Let $X$ be an $\infty$-category, we define the \emph{zero object} 0 to be an object of $X$ that is both initial and final, i.e.
$$\mbox{Map}_{X}(c,0)\simeq\mbox{Map}_{X}(0,c)\simeq\ast$$
for all $c\in X_0$.
\end{defn}

\begin{rem}
The zero object is unique up to equivalence.
\end{rem}

\begin{defn}[Pointed $\infty$-category]
We define a \emph{pointed $\infty$-category} to be an $\infty$-category equipped with a zero object.
\end{defn}
 
\begin{defn}[Fiber (cofiber) sequence] 
Let $X$ be a pointed $\infty$-category, we consider the functor of simplicial sets $T:\Delta^1\times\Delta^1\to X$ of the form:
\[
\xymatrix{
x\ar[d]\ar[r]^{f}&y\ar[d]^{g}\\
0\ar[r]&z
}
\]
We call $T$ a triangle in $X$. If $T$ is a pullback square we call it \emph{fiber sequence} (fiber of $g$), if $T$ is a pushout square we call it \emph{cofiber sequence} (cofiber of $f$).
\end{defn} 

\begin{rem} 
It easy to check that a triangle $T$ is the datum of:
\begin{itemize}
\item Two morphisms $f$, $g\in X_1$.
\item Two 2-simplices in $X_2$ of the form:
\[
\xymatrix{
x\ar[d]\ar[dr]^h&&&x\ar[r]^{f}\ar[dr]_{h}&y\ar[d]^{g} \\
0\ar[r]&z&&&z
}
\]
\end{itemize} 
We will indicate the \emph{triangle} $T$ by 
\[
\xymatrix{
x\ar[r]^f&y\ar[r]^{g}&z.
}
\]
\end{rem}

\begin{defn}[Stable $\infty$-category]
We say that $X$ is a \emph{stable $\infty$-category} if 
\begin{itemize}
\item[(S1)] $X$ is an $\infty$-category equipped with zero object (pointed $\infty$-category).
\item[(S2)] Every morphism has fibers and cofibers.
\item[(S3)] Every triangle in $X$ is a fiber sequence if and only if it is a cofiber sequence.
\end{itemize} 
\end{defn} 
Given a stable $\infty$-category $X$, we have an auto-equivalence $\Sigma:X\to X$ called \emph{suspension} functor, with inverse $\Omega$ called \emph{loop} functor, obtained via the category of subfunctors of $\mbox{Fun}(\Delta^1\times\Delta^1,X)$ generated by the following pullbacks and pushouts in $\Delta^1\times\Delta^1\to X$:
\[
\xymatrix{
x\ar[d]\ar[r]&0\ar[d]&&x_{\Omega}\ar[r]\ar[d]&0\ar[d] \\
0'\ar[r]&x_{\Sigma}&&0'\ar[r]&x
}
\]
where 0 and 0$'$ are zero objects in $X$ (cf. \cite[Chapter 1]{Lur2} for a precise definition). 
If $n>0$ we will denote by $x[n]$ the $\Sigma$ functor applied $n$-times to $x\in X$, if $n<0$ we will denote by $x[n]$ the $\Omega$ functor applied $n$-times to $x$.\\ 
We have the following fundamental theorem:

\begin{thm}[]
If $X$ is a stable $\infty$-category then the homotopy category \emph{$\mbox{Ho}(X)$} is a triangulated category with $\Sigma$ the $suspension$ functor as shift functor and distinguished triangles given by the following $\Delta^2\times\Delta^1\to X$ diagram:
\[
\xymatrix{
x\ar[d]\ar[r]&0\ar[d]\\
y\ar[d]\ar[r]&z\ar[d]\\
0'\ar[r]&w.
}
\]
\end{thm}

We denote by $\mbox{Cat}^{\tiny\mbox{St}}_\infty$ the category of \emph{stable $\infty$-categories} whose objects are the stable $\infty$-categories and whose morphisms are the functors of $\infty$-categories.\\

A functor between $\infty$-categories "a priori" does not give information about the zero object and the fiber sequences, so in the case of stable $\infty$-categories we prefer use the following definition of functors.

\begin{defn}[Exact functor]
Let $F:X\to X'$ be a functor between stable $\infty$-categories. We say that $F$ is \emph{exact} if the following are satisfied:
\begin{itemize}
\item[(E1)] $F(0_X)=0_{X'}$.
\item[(E2)] $F$ carries fiber sequences to fiber sequences.
\item[(E2$'$)] $F$ carries cofiber sequences to cofiber sequences.
\end{itemize}
\end{defn}
\begin{rem} 
If (E1) and (E2) holds true, than $F$ carries triangles to triangles. Moreover $F$ satisfies (E2) if and only if $F$ satisfies (E2$'$).
\end{rem}

\begin{exmp}
The identity functor of a stable $\infty$-category and the composition of two exact functors are exact functors. 
\end{exmp}

We denote by $\mbox{Cat}^{\tiny\mbox{Ex}}_\infty$ the \emph{exact stable $\infty$-category} whose objects are the stable $\infty$-categories and whose morphisms are the exact functors.

\subsection{$\mbox{A}_{\infty}$-nerve}\label{nervez}
The nerves are a useful tool to pass from a category to an $\infty$-category, in this section we will define the $\mbox{A}_{\infty}$-nerve, originally defined in \cite{Fao}, which is a generalization of the dg-nerve of Lurie. 

\begin{prp}
Let $n$ be a positive integer and $\mathscr{C}$ be an $\mbox{A}_{\infty}$-category. We denote by $[n]_{\mathbb{K}}$ the A$_{\infty}$-category such that $m^2_{[n]_{\mathbb{K}}}$ is given by (\ref{composiz}) and $m^{m\not=2}_{[n]_{\mathbb{K}}}=0$. 
Every maps {$\mathcal{f}\mathscr{F}^n\mathcal{g}\in\mbox{Hom}_{\tiny\mbox{A}_{\infty}\mbox{-Cat}}({[n]}_{\mathbb{K}},\mathscr{C})$} is uniquely determined by: 
\begin{itemize}
\item[1.] $n+1$-objects $\mathcal{f}X_{i}\mathcal{g}_{0\le i\le n}$ of $\mathscr{C}$, 
\item[2.] A set of morphisms $f_I$ for all set of integers $I=\mathcal{f}i_{0}<i_1<...<i_m<i_{m+1}\mathcal{g}$ where $0\le i_0<i_{m+1}\le n$ that satisfying the following:
\begin{align} \notag
 m_{\mathscr{C}}^1(f_{I})&=\sum_{1\le j\le m}(-1)^{j-1}f_{I-i_j}+\sum_{1\le j\le m}(-1)^{1+(m+1)(j-1)}m^2_{\mathscr{C}}(f_{i_j...i_{m+1}},f_{i_0...i_j})\\&+\sum_{r>2}\displaystyle\sum_{\ddagger_r}(-1)^{1+\epsilon_r}m^{r}_{\mathscr{C}}(f_{i_{m+1-s_r}...i_{m+1}},...,f_{i_0...i_{s_1}}). \label{eq:f}
\end{align} 
where
$$\ddagger_r=\mathcal{f}\mbox{$s_1$,..., $s_r\in\mathbb{N}$ $|$ $\displaystyle\sum_{j=1}^{r}{s_j=m+1}$}\mathcal{g}$$
$$\epsilon_r(i_1,...,i_r)=\sum_{2\le k\le r}(1-i_k+i_{k-1})i_{k-1}.$$
\end{itemize}
\end{prp}

\begin{proof}
Given an $\mbox{A}_{\infty}$-functor $\mathscr{F}=\mathcal{f}\mathscr{F}^{m}\mathcal{g}^{m\ge0}:\mbox{[$n$]}_{\mathbb{K}}\to\mathscr{C}$ the image of the map $\mathscr{F}_0$ is uniquely determined by $n+1$ objects $\mathcal{f}X_{i}\mathcal{g}_{0\le i\le n}$ in $\mathscr{C}$ because $\mbox{[$n$]}_{\mathbb{K}}$ has exactly $n+1$ objects.
Moreover fixed two integers $i_{-}$ and $i_{+}\in\mbox{[$n$]}_{\mathbb{K}}$ such that $i_{-}<i_{+}$, for every $0\le m\le n$ we consider the map:
$$\mathscr{F}^m:\mbox{Hom}_{\tiny\mbox{[$n$]}_{\mathbb{K}}}(i_{m-1},i_{+})\otimes...\otimes \mbox{Hom}_{\tiny\mbox{[$n$]}_{\mathbb{K}}}(i_{-},i_{1})\to \mbox{Hom}_{\mathscr{C}}\big(\mathscr{F}^0(X_-),\mathscr{F}^0(X_+)\big)[1-m]$$
the unique non-trivial ones are those such that $i_{-}<i_1<i_2<...<i_{m-1}<i_{m}<i_{+}$. So the image of $\mathscr{F}^m$ is non-zero if and only if we have a set $I$ of $m+1$-elements in [$n$]$_{\mathbb{K}}$ such that $I_{}=\mathcal{f}i_{-}<i_1<i_2<...<i_{m-1}<i_{m}<i_{+}\mathcal{g}$. Then $\mathscr{F}^m$ is uniquely determined by the image $f_I=\mathscr{F}^m(j_{i_{m-1}i_{+}},...,j_{i_{-}i_1})$ where $j_{kl}$ denotes the only one non trivial map in $\mbox{Hom}_{\tiny\mbox{[$n$]}_{\mathbb{K}}}(i_{k},i_{l})$, and clearly they satisfy \eqref{eq:f} because they are the image of the $\mbox{A}_{\infty}$-functor $\mathscr{F}$.
\end{proof}

\begin{prp}[]
\label{SS}
Given a map $\alpha:[m]_{\mathbb{K}}\to[n]_{\mathbb{K}}$ in $\Delta$, we have an induced map {$\mbox{Hom}_{\tiny\mbox{A}_{\infty}\mbox{-Cat}}(\alpha,\mathscr{C})$} given by: 
\begin{equation}
\begin{aligned}\notag
\mbox{Hom}_{\tiny\mbox{A}_{\infty}\mbox{-Cat}}(\alpha,\mathscr{C}):\mbox{Hom}_{\tiny\mbox{A}_{\infty}\mbox{-Cat}}({[n]}_{\mathbb{K}},\mathscr{C})&\to \mbox{Hom}_{\tiny\mbox{A}_{\infty}\mbox{-Cat}}({[m]}_{\mathbb{K}},\mathscr{C})\\
(\mathcal{f}X_i\mathcal{g}_{0\le i\le n},\mathcal{f}f_I\mathcal{g}\mathcal{g}) &\mapsto (\mathcal{f}X_{\alpha(j)}\mathcal{g}_{0\le j\le m},\mathcal{f}g_J\mathcal{g}\mathcal{g}).
\end{aligned}
\end{equation}
where $g_J$ is:
$$g_J=
\begin{cases} f_{\alpha(J)}, & \mbox{if $\alpha_{|J}$ is injective} \\ 1_{X_i}, & \mbox{if $J=\mathcal{f}j,j'\mathcal{g}$ and $\alpha(j)=\alpha(j')=X_i$} \\ 0, & \mbox{otherwise},
\end{cases}$$
such that, given $\alpha:{[m]}_{\mathbb{K}}\to{[n]}_{\mathbb{K}}$ and $\beta:{[n]}_{\mathbb{K}}\to{[l]}_{\mathbb{K}}$, then {$$\mbox{Hom}_{\tiny\mbox{A}_{\infty}\mbox{-Cat}}(\beta\cdot\alpha,\mathscr{C})=\mbox{Hom}_{\tiny\mbox{A}_{\infty}\mbox{-Cat}}(\alpha,\mathscr{C})\cdot\mbox{Hom}_{\tiny\mbox{A}_{\infty}\mbox{-Cat}}(\beta,\mathscr{C}).$$
Moreover given $\mbox{Id}:[n]_{\mathbb{K}}\to[n]_{\mathbb{K}}$ then $$\mbox{Hom}_{\tiny\mbox{A}_{\infty}\mbox{-Cat}}(\mbox{Id},\mathscr{C})=\mbox{Id}_{\footnotesize\mbox{Hom}_{\tiny\mbox{A}_{\infty}\mbox{-Cat}}([n],\mathscr{C})}.$$}
\end{prp}

\begin{proof}
First of all, we want to associate to $\alpha$ an $\mbox{A}_{\infty}$-functor (denoted by $\mathcal{f}\alpha\mathcal{g}$) between the minimal categories $\mbox{[$m$]}_{\mathbb{K}}\to\mbox{[$n$]}_{\mathbb{K}}$. 
We define the $\mbox{A}_{\infty}$-functor $\mathcal{f}\alpha^n\mathcal{g}^{n\ge0}:\mbox{[$m$]}_{\mathbb{K}}\to\mbox{[$n$]}_{\mathbb{K}}$ in the following way:
\begin{itemize}
\item if $k=0$, $\alpha^k=\alpha$, 
\item if $k=1$, \begin{equation}
\begin{aligned}\notag
\alpha^1:\mbox{Hom}_{\tiny\mbox{[$n$]}_{\mathbb{K}}}(l,s)&\to\mbox{Hom}_{\tiny\mbox{[$m$]}_{\mathbb{K}}}(\alpha(l),\alpha(s))\\
j_{ls}&\mapsto \alpha^1(j_{ls})=
\begin{cases} 0, & \mbox{if $l>s$} \\ 
1, & \mbox{if $l=s$} \\ 
j_{\alpha(l)\alpha(s)}, & \mbox{if $l<s$}
\end{cases}.
\end{aligned}
\end{equation}
\item if $k>1$, $\alpha^k=0$.
\end{itemize}
The induced map $\mbox{Hom}_{\tiny\mbox{A}_{\infty}\mbox{-Cat}}(\alpha,\mathscr{C})$ is given by the composition with the $\mbox{A}_{\infty}$-functor $\mathcal{f}\alpha^n\mathcal{g}_{n\ge0}$.
Let $\mathscr{F}\in\mbox{Hom}_{\tiny\mbox{A}_{\infty}\mbox{-Cat}}(\mbox{[$n$]}_{\mathbb{K}},\mathscr{C})$.\ 
For all $t\ge 1$ we have:
$$(\mathscr{F}\alpha)^t=\sum_{r=1}^t\sum_{i_1+...+i_r=t}\mathscr{F}^r(\alpha_{i_r},...,\alpha_{i_1}).$$
Since only $\alpha^1$ is non-trivial. We have $r=t$, $i_1=i_2=...=i_t=1$ and $(\mathscr{F}\alpha)^t$ becomes:
$$(\mathscr{F}\alpha)^t=\mathscr{F}^t(\alpha^{1},...,\alpha^{1}).$$
Therefore 
\begin{align*}
\mathscr{F}^1(\alpha_{1}(j_{i_0i_1}))&=\mathscr{F}^1(j_{\alpha(i_0)\alpha(i_1)})),\\
\mathscr{F}^2(\alpha_{1}(j_{i_0i_1}),\alpha_1(j_{i_1i_2}))&=\mathscr{F}^2(j_{\alpha(i_0)\alpha(i_1)},j_{\alpha(i_1)\alpha(i_2)})),\\ 
&...\\ 
\mathscr{F}^n(\alpha_{1}(j_{i_0i_1}),\alpha^1(j_{i_1i_2}),...,\alpha_{1}(j_{i_{n-1}i_n}))&=\mathscr{F}^n(j_{\alpha(i_0)\alpha(i_1)},j_{\alpha(i^1)\alpha(i_2)},...,j_{\alpha(i_{n-1}),\alpha(i_n)})).
\end{align*}
Of course, $i_k$ are positive integers smaller than $m$ (because $\alpha:\mbox{[$m$]}_{\mathbb{K}}\to\mbox{[$n$]}_{\mathbb{K}}$), so if we take an element in $\mbox{Hom}_{\tiny\mbox{A}_{\infty}\mbox{-Cat}}(\mbox{[$n$]}_{\mathbb{K}},\mathscr{C})$ denoted by $(\mathcal{f}X_i\mathcal{g}_{0\le i\le n},\mathcal{f}f_I\mathcal{g}\mathcal{g})$ this is sent to $(\mathcal{f}X_{\alpha(j)}\mathcal{g}_{0\le j\le m},\mathcal{f}g_J\mathcal{g}\mathcal{g})$ where $g_J$ is:
$$g_J=
\begin{cases} f_{\alpha(J)}, & \mbox{if $\alpha_{|J}$ is injective} \\ 1_{X_i}, & \mbox{if $J=\mathcal{f}j,j'\mathcal{g}$ and $\alpha(j)=\alpha(j')=X_i$} \\ 0, & \mbox{otherwise}
\end{cases}$$
and we are done.
\end{proof}

\begin{defn}[$\mbox{A}_{\infty}$-nerve]
Let $\mathscr{C}$ be an $\mbox{A}_{\infty}$-category. We define the \emph{$\mbox{A}_{\infty}$-nerve} of $\mathscr{C}$ to be the simplicial set (denoted by $\mbox{N}_{\tiny\mbox{A}_{\infty}}(\mathscr{C}$)) such that for all positive integers $n$
$$\mbox{N}_{\tiny\mbox{A}_{\infty}}(\mathscr{C})_n:=\mbox{Hom}_{\tiny\mbox{A}_{\infty}\mbox{-Cat}}(\mbox{[$n$]}_{\mathbb{K}},\mathscr{C}).$$
For every $\alpha:\mbox{[$m$]}\to\mbox{[$n$]}\in\Delta$ the element $(\mathcal{f}X_i\mathcal{g}_{0\le i\le n},\mathcal{f}f_I\mathcal{g}\mathcal{g})$ in $\mbox{N}_{\tiny\mbox{A}_{\infty}}(\mathscr{C})_n$ is sent to $(\mathcal{f}X_{\alpha(j)}\mathcal{g}_{0\le j\le m},\mathcal{f}g_J\mathcal{g}\mathcal{g})$ where $g_J$ is:
$$g_J=
\begin{cases} f_{\alpha(J)}, & \mbox{if $\alpha_{|J}$ is injective} \\ 1_{X_i}, & \mbox{if $J=\mathcal{f}j,j'\mathcal{g}$ and $\alpha(j)=\alpha(j')=X_i$} \\ 0, & \mbox{otherwise}.
\end{cases}$$
\end{defn}

\begin{rem}
\label{NN}
Note that if $\mathscr{C}$ is a dg-category then $\mbox{N}_{\tiny\mbox{A}_{\infty}}(i(\mathscr{C}))=\mbox{N}_{\tiny\mbox{dg}}(\mathscr{C})$ where $\mbox{N}_{\tiny\mbox{dg}}$ is the dg-nerve defined in \cite[\S 1.3.1.6]{Lur1}.
\end{rem}

\begin{thm}
Let $\mathscr{C}$ be an $\mbox{A}_{\infty}$-category, then {$\mbox{N}_{\tiny{\mbox{A}}_{\infty}}(\mathscr{C})$} is an $\infty$-category.\end{thm}

\begin{proof}
\cite[Proposition 2.2.12]{Fao}. 
\end{proof}

\section{Properties of the $\mbox{A}_{\infty}$-nerves}

This section is divided in three parts: in the first one we will give a useful characterization of the mapping space of the $\mbox{A}_\infty$-nerve, in the second we will recall some classical result about model categories, finally we will prove the main theorem of the paper that will be the fundamental tool to give a comparison between the $\mbox{A}_\infty$-categories and the stable $\infty$-categories.

\subsection{Simplicial Objects and DK-correspondence}
Let $\mathcal{A}$ be an abelian category, we denote by $\mbox{Ch}^{\tiny\ge 0}_{\mathcal{A}}$ the category of chain complexes bounded above. In particular if $\mathcal{A}$ is the category of $\mathbb{K}$-modules, we denote by $\mbox{Ch}^{\tiny\ge 0}_{\mathbb{K}}$ the the category of chain complexes of $\mathbb{K}$-modules bounded above.

\begin{defn}[Simplicial Object]
A \emph{simplicial object} $A$ in $\mathcal{A}$ is a functor $A:\Delta^{\tiny\mbox{op}}\to\mathcal{A}$.
\end{defn}

We have a functor $\textbf{N}_{*}:\mbox{Fun}(\Delta^{\tiny\mbox{op}},\mathcal{A})\to \mbox{Ch}^{\tiny\ge 0}_{\mathcal{A}}$ that associates to each simplicial object $A_{\cdot}$ the chain:
\[
\xymatrix{
...\ar[r]&\textbf{N}_2(A)\ar[r]^{A(d_0)}&\textbf{N}_1(A)\ar[r]^{A(d_0)}&\textbf{N}_0(A)\ar[r]&0\ar[r]&...
}
\] 
where:
$$\textbf{N}_n(A_{\cdot}):=\bigcap_{1\le i\le n}\mbox{ker}(A(d_i))$$and $d_j:\mbox{[$n-1$]}\to\mbox{[$n$]}$ is the natural injective map such that $j\not\in\mbox{Im($d_j$)}$.\\

We have also a functor {$\mbox{DK}_\bullet:\mbox{Ch}^{\tiny\ge0}_\mathcal{A}\to\mbox{Fun}(\Delta^{\tiny\mbox{op}},\mathcal{A})$} that associates to each chain ${C}^{\bullet}$ the simplicial object $\mbox{DK}_*(C):\Delta^{\tiny\mbox{op}}\to\mathcal{A}$ defined, for every $n$, to be:
$$\mbox{DK}_n(C):=\displaystyle\bigoplus_{\alpha:\tiny\mbox{[$n$]$\to$[$k$]}}C_k,$$ where $\alpha$ is a surjective map.\\
Moreover, given a map $\beta:\mbox{[$n'$]$\to$[$n$]}$, we define $\mbox{DK}_{\bullet}(\beta)$ to be the matrix with $(\alpha,\alpha')$ entries: $$(f_{\alpha,\alpha'}):\displaystyle\bigoplus_{\alpha}C_k\to\displaystyle\bigoplus_{\alpha'}C_{k'}$$ such that:
$$f_{\alpha,\alpha'}=\begin{cases} 1_{C_k}, & \mbox{if $\alpha$ and $\alpha'$ are fit in a diagram 
\xymatrix{
[n]\ar[r]^{\beta}\ar[d]_{\alpha}&[n']\ar[d]^{\alpha'}\\
[k']\ar@_{=}[r]&[k]
}
} \\ 
d_k, & \mbox{if $\alpha$ and $\alpha'$ are fit in a diagram \xymatrix{
[n]\ar[d]_{\alpha}\ar[r]^{\beta}&[n']\ar[d]^{\alpha'}\\
[k-1]\ar[r]_{d_0}&[k]
}} 
\\ 0, & \mbox{otherwise}.
\end{cases}$$

\begin{thm}[]
\label{DP}
The functors $\mathrm{DK}_\bullet$, {$\textbf{{N}}_{*}$} are adjoints in both directions (i.e. $\mathrm{DK}_\bullet\vdash$ {$\textbf{N}_{*}$} and {$\textbf{N}_{*}$}$\vdash\mathrm{DK}_\bullet$):
\end{thm}

\begin{proof} 
\cite[Satz 3.6]{DoPu}.
\end{proof}

Let $\textbf{Z}\Delta^n$ denote the free abelian group generated by $\Delta^n[j]$, for every $j$. Let us build the chain associated $\textbf{N}_{*}(\textbf{Z}\Delta^n)$.

\begin{exmp}
\label{Bombolino}
We take $\Delta^0=\mbox{Hom}_{\Delta}(-,\mbox{[$0$]})$, if $n=0$ then $\textbf{N}_0(\textbf{Z}\Delta^0)=\mbox{ker}(\textbf{Z}\Delta^0_0\to 0)=\textbf{Z}\Delta^0_0=\mathcal{f}\mbox{1 generator $g_0$}\mathcal{g}$. If $n=1$, by definition, $\textbf{N}_1(\textbf{Z}\Delta^0)=\mbox{ker}(d^1:\textbf{Z}\Delta^0_1\to \textbf{Z}\Delta^0_0)=0$, because $\textbf{Z}\Delta^0_1$ is generated by $g_{00}$ and $d^1(g_{00})=g_0\not=0$. We can procede in the same way for all the other $n\ge 1$. Hence the chain associated to $\textbf{Z}(\Delta^0)$ is given by: 
\[
\xymatrix{
...\ar[r]&0\ar[r]^{d_0}&0\ar[r]^{d_0}&<g_0>\ar[r]&0\ar[r]&...
}
\] 
\end{exmp}

\begin{exmp}
\label{bombolone}
We take $\Delta^1=\mbox{Hom}_{\Delta}(-,\mbox{[$1$]})$, if $n=0$ we have $\textbf{N}_0(\textbf{Z}\Delta^1)=\mbox{ker}(\textbf{Z}\Delta^1_0\to 0)=\textbf{Z}\Delta^1_0=\mathcal{f}\mbox{2 generators $g_0$ and $g_1$}\mathcal{g}$. If $n=1$ we have $\textbf{N}_1(\textbf{Z}\Delta^1)=\mbox{ker}(d^1:\textbf{Z}\Delta^1_1\to \textbf{Z}\Delta^1_0)$. In $\textbf{Z}\Delta^1_1$ we have three generators $g_{00}$, $g_{01}$ and $g_{11}$ given by the following maps:\\
\[
\xymatrix@R=2mm@C=8mm{
0\ar@{->}[rr]&&0\\
1\ar@{->}[urr]&&1\\
&g_{00}&
}\mbox{              ,             }
\xymatrix@R=2mm@C=8mm{
0\ar@{->}[rr]&&0\\
1\ar@{->}[rr]&&1\\
&g_{01}&
}\mbox{    ,    }
\xymatrix@R=2mm@C=8mm{
0\ar@{->}[drr]&&0\\
1\ar@{->}[rr]&&1\\
&g_{11}&
}
\]

$\textbf{N}_1(\textbf{Z}\Delta^1)$ is given by the elements $\textbf{Z}\Delta^1_1$ of the form $\alpha_{00}g_{00}\oplus\alpha_{01}g_{01}\oplus\alpha_{11}g_{11}$ such that $d^1=0$, where $\alpha_{ij}\in\mathbb{K}$. By definition:

\begin{equation}
\begin{split}
d^1(\alpha_{00}g_{00}\oplus\alpha_{01}g_{01}\oplus\alpha_{11}g_{11})&=\alpha_{00}g_{0}\oplus\alpha_{01}g_{0}\oplus\alpha_{11}g_{1}\\
&=(\alpha_{00}+\alpha_{01})g_{0}\oplus\alpha_{11}g_{1}.
\end{split}
\end{equation}
and it is zero only if $\alpha_{00}+\alpha_{01}=0$ and $\alpha_{11}=0$.\\ 
Hence $\mbox{ker}(\textbf{Z}\Delta^1_1\to \textbf{Z}\Delta^1_0)=<g_{00}-g_{01}>$.\\
Then the associated chain $\textbf{Z}(\Delta^1)$ is given by: 
\[
\xymatrix{
...\ar[r]&0\ar[r]&<g_{00}-g_{01}>\ar[r]^{d_0}&<g_0>\oplus<g_1>\ar[r]&0\ar[r]&...
}
\] 
such that $d^0<g_{00}-g_{01}>=g_{0}-g_{1}$
\end{exmp}

Let  $\mathscr{C}$ be a dg-category and $x$, $y$ two fixed objects in $\mathscr{C}$.
By {Example \ref{Bombolino}} we can identify the homomorphisms of complexes $f:\textbf{N}_{*}(\textbf{Z}\Delta^0)\to\mbox{Hom}_{\mathscr{C}}(x,y)$ with the maps $f:x\to y$ of degree zero such that $df=0$. By {Example \ref{bombolone}}, we can identify the homomorphisms of complexes $f:\textbf{N}_{*}(\textbf{Z}\Delta^1)\to\mbox{Hom}_{\mathscr{C}}(x,y)$ with the set of the maps $f_{02}, f_{12},f_{012}:x\to y$ such that $\mbox{deg }f_{02}=\mbox{deg }f_{12}=0$, $\mbox{deg }f_{012}=-1$, $df_{012}=f_{02}-f_{12}$ and $df_{02}=df_{12}=0$.\\
More generally let us discuss an important lemma (implicitly assumed by Lurie \cite[pg. 66]{Lur2}) which characterizes the maps between $\textbf{N}_{*}(\textbf{Z}\Delta^n)$ and $\mbox{Hom}_{\mathscr{C}}(x,y)$.

\begin{lem}
\label{comb}
We can identify {$f:\textbf{N}_{*}(\textbf{Z}\Delta^n)\to\mbox{Hom}_{\mathscr{C}}(x,y)$} to the maps $f_{I}:x\to y$ of degree $|I|-2$ for all subset $I=\mathcal{f}0\le i_0< ...< i_j<j+1\le n\mathcal{g}$ such that:
\begin{equation}
df_I=\sum_{0\le k\le j}(-1)^{k}f_{I-k}. \tag{$\dagger$}
\end{equation}
\end{lem}

\begin{proof}
We denote by $g_{i_0...i_j}$ the free generator associated to the map $\mbox{[$j$]}\to\mbox{[$n$]}$ which sends the integer $k\in\mbox{[$j$]}$ to $i_k\in\mbox{[$n$]}$. It follows immediately that
$$\langle\displaystyle\bigoplus_{0\le i_0\le...\le i_j\le n}g_{i_0...i_j}\rangle=\textbf{Z}\Delta^n_j.$$ 
By definition, an element $\displaystyle\bigoplus_{0\le i_0\le...\le i_j\le n}\alpha_{i_0...i_j}g_{i_0...i_j}$ is in $\textbf{N}_j(\textbf{Z}\Delta^n)$ if and only if
\begin{equation}
\label{dj}
\left\{
\begin{split}
d^{j}(\displaystyle\bigoplus_{0\le i_0\le...\le i_j\le n}\alpha_{i_0...i_j}g_{i_0...i_j})&=0\\ 
...\\
d^{1}(\displaystyle\bigoplus_{0\le i_0\le...\le i_j\le n}\alpha_{i_0...i_j}g_{i_0...i_j})&=0
\end{split}
\right.
\end{equation}
Now, if we focus on the first row in (\ref{dj}), we have that

\begin{equation}
\begin{split}
\label{riga1}
d^{j}(\displaystyle\bigoplus_{0\le i_0\le...\le i_j\le n}\alpha_{i_0...i_j}g_{i_0...i_j})=0
\end{split}
\end{equation}
if and only if 
\begin{equation*}
\begin{split}
\displaystyle\sum_{i_j=i_{j-1}+1}^n\alpha_{i_0...i}=-\alpha_{i_0...i_{j-1}i_{j-1}}.
\end{split}
\end{equation*}
So we can rewrite (\ref{dj}) in terms of the following system of $j-1$ equations
\begin{equation}
\label{dj-1}
\left\{
\begin{split}
d^{j-1}(\displaystyle\bigoplus_{0\le i_0\le...\le i_j\le n}\alpha_{i_0...i_j}(g_{i_0i_1...i_{j}}-g_{i_0...i_{j-1}i_{j-1}})
&=0\\
...\\
d^{1}(\displaystyle\bigoplus_{0\le i_0\le...\le i_j\le n}\alpha_{i_0...i_j}(g_{i_0i_1...i_{j}}-g_{i_0...i_{j-1}i_{j-1}})&=0.
\end{split}
\right.
\end{equation}
Proceeding as for the first row, we obtain the following system of $j-2$ equations equivalent to (\ref{dj-1})

\begin{equation}
\label{soad2}
\left\{
\begin{split}
d^{j-2}(\displaystyle\bigoplus_{0\le i_0\le...\le i_j\le n}\alpha_{i_0...i_j}(g_{i_0i_1...i_{j}}-g_{i_0...i_{j-2}i_{j-2}i_j}+\\-(g_{i_0i_1...i_{j-1}}-g_{i_0...i_{j-2}i_{j-2}i_{j-1}}))
&=0\\
...\\
...\\
d^{1}(\displaystyle\bigoplus_{0\le i_0\le...\le i_j\le n}\alpha_{i_0...i_j}(g_{i_0i_1...i_{j}}-g_{i_0...i_{j-2}i_{j-2}i_j}+\\-(g_{i_0i_1...i_{j-1}}-g_{i_0...i_{j-2}i_{j-2}i_{j-1}}))
&=0.
\end{split}
\right.
\end{equation}
We can go on as before by removing one by one the equations from the system. In the end we have that $\displaystyle\bigoplus_{0\le i_0\le...\le i_j\le n}\alpha_{i_0...i_j}g_{i_0...i_j}$ is in $\textbf{N}_j(\textbf{Z}\Delta^n)$ if it is of the form
\begin{equation*}
\begin{split}
\displaystyle\bigoplus_{0\le i_0\le...\le i_j\le n}\alpha_{i_0...i_j}(\displaystyle\sum_{0\le k^0_1,...,k^{j-1}_j \le 1}(-1)^{\triangle_{{i}^{k^0_1}_1...{i}^{k^{j-1}_j}_j}}g_{i_0{i}^{k^0_1}_1...{i}^{k^{j-1}_j}_j})
\end{split}
\end{equation*}
where
$${i}^{k^{l_1}_{l_2}}_l=
\begin{cases} 
i_{l_2}, \mbox{ if $k^{l_1}_{l_2}=0$}\\
i_{l_1}, \mbox{ if $k^{l_1}_{l_2}=1$}
\end{cases}
$$and
$$\triangle_{{i}^{k^0_1}_1...{i}^{k^{j-1}_j}_j}=k^0_1+...+k^{j-1}_j.$$
We note that, if there exists $p$ such that $i_p=i_{p-1}$, then 
\begin{equation*}
\begin{split}
\displaystyle\sum_{0\le k^0_1,...,k^{j-1}_j \le 1}(-1)^{\triangle_{{i}^{k^0_1}_1...{i}^{k^{j-1}_j}_j}}g_{i_0{i}^{k^0_1}_1...{i}^{k^{j-1}_j}_j}=0.
\end{split}
\end{equation*}
This means that $\textbf{N}_j(\textbf{Z}\Delta^n)=0$ if $j>n$. Otherwise $\textbf{N}_j(\textbf{Z}\Delta^n)$ is generated by 
\begin{equation}
\label{tu}
\begin{split}
\displaystyle\bigoplus_{0\le i_0<...< i_j\le n}(\displaystyle\sum_{0\le k^0_1,...,k^{j-1}_j \le 1}(-1)^{\triangle_{{i}^{k^0_1}_1...{i}^{k^{j-1}_j}_j}}g_{i_0{i}^{k^0_1}_1...{i}^{k^{j-1}_j}_j}).
\end{split}
\end{equation}
Now, every map of complexes $f:\textbf{N}_{*}(\textbf{Z}\Delta^n)\to\mbox{Hom}_{\mathscr{C}}(x,y)$ is uniquely determined, for every integer $j$, by the image of the generators in (\ref{tu}). We will denote by $f_{i_0...i_j(j+1)}$ such images.
Moreover $f$ is a chain of complexes. So 
\begin{equation}
\label{differentialcool}
\begin{split}
d^j(f_{i_0...i_j(j+1)})&=f_{j-1}(\displaystyle\sum_{0\le k^0_1,...,k^{j-1}_j \le 1}(-1)^{\triangle_{{i}^{k^0_1}_1...{i}^{k^{j-1}_j}_j}}g_{i^{k^0_1}_1i^{k^1_2}_2...{i}^{k^{j-1}_j}_j}))\\
&=f_{j-1}(\displaystyle\sum_{0\le k^1_2,...,k^{j-1}_j \le 1}(-1)^{\triangle_{{i}^{k^1_2}_2...{i}^{k^{j-1}_j}_j}}(g_{i_1i^{k^1_2}_2...{i}^{k^{j-1}_j}_j}-g_{i_0i^{k^1_2}_2...{i}^{k^{j-1}_j}_j}))\\
&=f_{i_1...i_j(j+1)}-f_{j-1}(\displaystyle\sum_{0\le k^1_2,...,k^{j-1}_j \le 1}(-1)^{\triangle_{{i}^{k^1_2}_2...{i}^{k^{j-1}_j}_j}}(g_{i_0i^{k^1_2}_2...{i}^{k^{j-1}_j}_j})).
\end{split}
\end{equation}
Note that, for every $t$, we have
\begin{equation*}
\begin{split}
g_{i^{k^0_1}_1i^{k^1_2}_2...i^{k^{t-3}_{t-2}}_{t-2}i^{k^{t-1}_t}_ti^{k^{t}_{t+1}}_{t+1}...{i}^{k^{j-1}_j}_j}&=g_{i^{k^0_1}_1i^{k^1_2}_2...i^{k^{t-3}_{t-2}}_{t-2}i_t...{i}^{k^{j-1}_j}_j}-g_{i^{k^0_1}_1i^{k^1_2}_2...i^{k^{t-3}_{t-2}}_{t-2}i_{t-1}...{i}^{k^{j-1}_j}_j}\\
&=g_{i^{k^0_1}_1i^{k^1_2}_2...i^{k^{t-3}_{t-2}}_{t-2}i^{k^{t-2}_{t}}_ti^{k^{t}_{t+1}}_{t+1}...{i}^{k^{j-1}_j}_j}+\\
&-g_{i^{k^0_1}_1i^{k^1_2}_2...i^{k^{t-3}_{t-2}}_{t-2}i^{k^{t-2}_{t-1}}_{t-1}...{i}^{k^{j-1}_j}_j}.
\end{split}
\end{equation*}
This means that equation (\ref{differentialcool}) gives precisely the condition ($\dagger$).
\end{proof}

\begin{rem}
\label{LS}
By Theorem \ref{DP} we have that $$\mbox{Hom}_{}(\textbf{Z}\Delta^n,\mbox{DK}_{\bullet}(\tau_{\ge0}\mbox{Hom}_{\mathscr{C}}(x,y)))\simeq\mbox{Hom}_{\tiny\mbox{Ch}_{\mathbb{K}}}(\textbf{N}_*(\textbf{Z}\Delta^n),\tau_{\ge0}\mbox{Hom}_{\mathscr{C}}(x,y)).$$Using the characterization in Lemma \ref{comb} we have that the morphisms $f_I$ with the property $(\dagger)$ are in bijection with $\mbox{DK}_{n}(\tau_{\ge0}\mbox{Hom}_{\mathscr{C}}(x,y))$.
\end{rem}

\subsection{Model structures}
We briefly recall some classical notions about model structures on categories. A good reference about model structures for the beginners is \cite{Hov}.

\begin{exmp}
\label{T}
The category of (small) dg-categories has two canonically model structures due to Tabuada \cite{Tab1} \cite{Tab2}: the first one has as weak-equivalences the "classical" quasi-equivalences and the second one has as weak-equivalences the Morita equivalences. We recall that $F:\mathscr{C}\to\mathscr{C}'$ is a Morita equivalence if:
\begin{itemize}
\item[(Me1)] $F$ induces an equivalence on perfect-complexes $$\mbox{Ho}(F):\mbox{Ho}(\mbox{pretr}(\mathscr{C}))^{ic}\to\mbox{Ho}(\mbox{pretr}(\mathscr{C}'))^{ic}$$
\item[(Me2)] $\mbox{Hom}_{\mathscr{C}}(x,y)\to\mbox{Hom}_{\mathscr{C}'}(F(x),F(y))$ is a quasi-isomorphism for all $x$, $y\in\mathscr{C}$.
\end{itemize}
Clearly every weak equivalence in the first model structure is a Morita equivalence.  
\end{exmp}
\begin{rem}
A functor between pretriangulated idempotent complete dg-categories is a weak-equivalence if and only if it is a Morita equivalence.
\end{rem}

\begin{exmp}\label{connective}
The category of connective (i.e non negative) chain complexes $\mbox{Ch}^{\ge0}_{\bullet}(R)$ has a model structure such that:
\begin{itemize}
\item[1.] Weak-equivalences are the quasi-isomorphisms.
\item[2.] The cofibrations are the morphisms degreewise injectives with degreewise projective cokernels.
\item[3.] The fibrations are the morphisms degreewise surjective in positive degree.
\end{itemize}
Note that, with this model structure, all the objects are fibrant.
\end{exmp}

\begin{defn}[Weak equivalence \cite{Joy}]\label{stork}
Let $X$, $Y$ be $\infty$-categories, $F:X\to Y$ is a \emph{weak equivalence} if:
\begin{itemize}
\item[1.] $\mbox{Ho}(X)\simeq\mbox{Ho}(Y)$ (as categories),
\item[2.] $\forall x,y\in X$ the geometric realization of the morphism $$\mbox{Hom}^R_{X}(x,y)\to \mbox{Hom}^R_{Y}(F_0(x),F_0(y))$$ is a weak homotopy equivalence of topological spaces.
\end{itemize}
\end{defn}
Weak equivalences together with monomorphisms (i.e. $F_n:X_n\to Y_n$ monomorphisms for all $n>0$) as cofibrations and fibrations, defined by the right left property (cf.\ \cite[Definition 1.1.2]{Hov}), forms a model structure over sSet called \emph{Joyal model structure}. 
\begin{rem}
We can see a simplicial object as a simplicial set by applying the forgetful functor. 
\end{rem} 
Using \cite[Theorem 4]{Qui} we can endow the category of simplicial objects with a model structure defining weak equivalences (resp.\ fibrations) as the morphisms of simplicial objects where the underling functor is a weak equivalence (resp.\ Kan fibrations) of simplicial sets.

\begin{rem}
Let $x$, $y\in\mathscr{C}$ where $\mathscr{C}$ is a dg-category.\ 
There is an isomorphism of simplicial sets
\begin{align*}
\mbox{Hom}^R_{\footnotesize\mbox{N}_{\tiny\mbox{dg}}}(x,y)\simeq \mbox{DK}_\bullet(\tau_{\ge0}\mbox{Hom}_\mathscr{C}(x,y)),
\end{align*}
see \cite[Remark 1.3.1.12]{Lur2}.
\end{rem}

\begin{rem}
\label{DKW}
The functors $\mbox{DK}_\bullet$ and $\textbf{N}_{*}$ match cofibrations, fibrations and weak equivalences in the model structures on $\mbox{Ch}^{\ge 0}_{\mathbb{K}}$ (see Example \ref{connective}) in the above model structure over the simplicial objects $\mbox{Fun}(\Delta^{\tiny\mbox{op}},\mbox{$\mathbb{K}$-Mod})$ \cite[\S 4.1]{SS}. 
\end{rem}

\subsection{Main results}
Now we are ready to prove some new results about $\mbox{A}_{\infty}$-nerves that will be useful to give a comparison between pretriangulated $\mbox{A}_{\infty}$-categories and stable $\infty$-categories in the last section. Let $X$ be a simplicial set and let $x$, $y$ be two elements in $X_0$.

\begin{defn}[Degenerate simplex]
We define the \emph{degenerate $n$-simplex} on $x$ to be the image of $x$ via $X(\sigma)$, where $\sigma:\mbox{[$n$]}\to\mbox{[0]}$.
\end{defn}

\begin{exmp}
A degenerate 2-simplex on $x$ in $\mbox{N}_{\tiny{\mbox{A}}_{\infty}}(\mathscr{C})$ is represented by the following diagram:
\[
\xymatrix{
& x\ar[dr]^{1_x} &\\
x\ar[ur]^{1_x}\ar[rr]^{1_x}\ar@/_/[rr]_0 &&x
}
\]
\end{exmp}

\begin{defn}[Mapping space] 
For every couple of elements of $\mathscr{C}$, we define the \emph{mapping space} $\mbox{Hom}^R_X(x,y)$ to be the $\infty$-category whose $n$-simplexes are the $n+1$-simplexes of $X_{n+1}$ such that $X_{|_{\mathcal{f}n+1\mathcal{g}}}=y$ and $X_{|_{\mathcal{f}0,...,n\mathcal{g}}}$ is the degenerate $n$-simplex on $x$.
\end{defn}

\begin{lem}
\label{C}
Let $\mathscr{C}$ be an $\mbox{A}_{\infty}$-category. The mapping space {{$\mbox{Hom}^R_{\footnotesize\mbox{N}_{\tiny\mbox{A}_{\infty}}(\mathscr{C})}(x,y)$}} is equivalent (as simplicial set) to {$\mbox{DK}_{\bullet}$}$(\tau_{\ge 0}\mbox{Hom}_{\mathscr{C}}(x,y))$.
\end{lem}

\begin{proof}
First of all we compute the degenerate $n$-simplex in $\mbox{N}_{\tiny{\mbox{A}}_{\infty}}(\mathscr{C})$. Let us consider the degenerate map $\sigma:\mbox{[$n$]}\to\mbox{[$0$]}$. Using Theorem \ref{SS}, the image of $x$ in $\mbox{N}_{\tiny
{\mbox{A}}_{\infty}}(\mathscr{C})_n$ via $\mbox{N}_{\tiny\mbox{A}_{\infty}}(\sigma)$ is given by:
\begin{itemize}
\item $n+1$-copies of $x$, because $\alpha(i_0)=...=\alpha(i_n)=0$; 
\item identity maps between $x$, because $\alpha(j_{i_0i_1})=1_{X_{i_0}}$;
\item all the higher maps $f_{i_0i_1i_2}$,... are zeroes, because [$0$] has only one object.
\end{itemize}
By definition we have that, for every integer $n$, $\mbox{Hom}^R_{\footnotesize\mbox{N}_{\tiny\mbox{A}_{\infty}}(\mathscr{C})}(x,y)_n\subset\mbox{N}_{\footnotesize\mbox{A}_{\infty}}(\mathscr{C})_{n+1}$. Then an element of $\mbox{Hom}^R_{\footnotesize\mbox{N}_{\tiny\mbox{A}_{\infty}}(\mathscr{C})}(x,y)_n$ is a set of elements satisfying \eqref{eq:f} for all sets $I=\mathcal{f}0\le i_0<i_1<...<i_{m}<i_{m+1}\le n+1\mathcal{g}$.\\
Now, using the previous calculation on degenerate $n$-simplex, we have that every $f_{i_pi_q}$ with $i_q\not=n+1$ is the identity and every $f_{i_p...i_q}$, with $q\not=n+1$, is $0$.\\
Then we can say that every element in $\mbox{Hom}^R_{\footnotesize\mbox{N}_{\tiny\mbox{A}_{\infty}}(\mathscr{C})}(x,y)_n$ is given by the identity maps on the vertex $x$ and, for all subsets $I=\mathcal{f}0\le i_0<i_1<...<i_{m}<i_{m+1}=n+1\mathcal{g}$, the maps $f_I$ (i.e. the maps with target $y$) satisfy:
$$m_{\mathscr{C}}^1(f_{I})=\sum_{1\le j\le m}(-1)^{j-1}(f_{I-i_j})-(-1)^{0}m^2_{\mathscr{C}}(f_{i_1...i_{m+1}},f_{i_0i_1})+\sum_{r>2}\sum_{\ddagger_{r}}(-1)^{1+\epsilon_r}0.$$ 
This means that
\begin{equation} \notag
\begin{split}
m_{\mathscr{C}}^1(f_{I})&=\sum_{1\le j\le m}(-1)^{j-1}(f_{I-i_j})-(-1)^{0}m^2_{\mathscr{C}}(f_{i_1...i_{m+1}},f_{i_0i_1})+\sum_{r>2}\sum_{\ddagger_{r}}(-1)^{1+\epsilon_r}0\\
&=-f_{i_1...i_{m+1}}+\sum_{1\le j\le m}(-1)^{j-1}(f_{I-i_j})\\
&=\sum_{0\le j\le m}(-1)^{j+1}(f_{I-i_j})
\end{split}
\end{equation}
Hence, after a change of sign, all the maps in $\mbox{Hom}^R_{\footnotesize\mbox{N}_{\tiny{\mbox{A}}_{\infty}}(\mathscr{C})}(x,y)$ satisfy $(\dagger)$ so, using Remark \ref{LS} and Theorem \ref{DP}, we have an isomorphism 
\begin{align*}
\mbox{Hom}^R_{\footnotesize\mbox{N}_{\tiny{\mbox{A}}_{\infty}}(\mathscr{C})}(x,y)\simeq\mbox{DK}_{\bullet}(\tau_{\ge 0}\mbox{Hom}_{\mathscr{C}}(x,y))
\end{align*}
and we are done.
\end{proof}

\begin{thm}[]
\label{M}
Let $\mathscr{C}$ and $\mathscr{D}$ be $\mbox{A}_{\infty}$-categories and let $\mathscr{F}:\mathscr{C}\to\mathscr{D}$ be a quasi-equivalence of $\mbox{A}_{\infty}$-categories.\ Then {$\mbox{N}_{\tiny{\mbox{A}}_{\infty}}(\mathscr{F}):\mbox{N}_{\tiny{\mbox{A}}_{\infty}}(\mathscr{C})\to\mbox{N}_{\tiny{\mbox{A}}_{\infty}}(\mathscr{D})$} is an weak-equivalence in the Joyal model structure.
\end{thm}

\begin{proof}
If $\mathcal{f}\mathscr{F}^n\mathcal{g}$ is a quasi-equivalence then, by definition, the functor induced between the homotopy category $\mbox{Ho}(\mathscr{C})$ and $\mbox{Ho}(\mathscr{D})$ is an equivalence (we1). We observe that the homotopic category of an $\infty$-category $X$ is given by the category having as objects the elements of $X_0$ and as morphisms the elements of $X_1$ that are quotient by the homotopy relation. So $\mbox{Ho}(\mbox{N}_{\tiny{\mbox{A}}_{\infty}}(\mathscr{C}))$ has the same objects as $\mathscr{C}$ and as morphisms the set $Z^0(\mbox{Hom}_{\mathscr{C}}(x,y))$ such that $f\simeq g$ if and only if there exists $h\in\mbox{Hom}_{\mathscr{C}}(x,y)^{-1}$ such that $dh=f-g$. It follows that $\mbox{N}_{\tiny{\mbox{A}}_{\infty}}(\mathscr{F})$ induces an equivalence between the homotopy categories of $\mbox{N}_{\tiny{\mbox{A}}_{\infty}}(\mathscr{C})$ and $\mbox{N}_{\tiny{\mbox{A}}_{\infty}}(\mathscr{D})$.\\
Now we have to prove that, given two objects $x$, $y\in\mathscr{C}$, the map
\begin{equation}
\label{WE}
\mbox{Hom}^R_{\footnotesize{\mbox{N}_{\tiny{\mbox{A}}_{\infty}}(\mathscr{C})}}(x,y)\to \mbox{Hom}^R_{\footnotesize{\mbox{N}_{\tiny{\mbox{A}}_{\infty}}(\mathscr{C})}}(\mathscr{F}^0(x),\mathscr{F}^0(y))
\end{equation}
is an homotopy equivalence between the corresponding Kan complex.\ 
Using Lemma \ref{C}, we have that it is enough to prove that 
\begin{equation}
\mbox{DK}_{\bullet}(\tau_{\ge0}\mbox{Hom}_{\mathscr{C}}(x,y))\to\mbox{DK}_{\bullet}(\tau_{\ge0}\mbox{Hom}_{\mathscr{D}}(\mathscr{F}^0(x),\mathscr{F}^0(y)))
\end{equation}
is a weak equivalence, and this is true because the functor $\mbox{DK}_{\bullet}$ preserves weak equivalences and the map of complexes $\mbox{Hom}_{\mathscr{C}}(x,y)\to\mbox{Hom}_{\mathscr{D}}(\mathscr{F}^0(x),\mathscr{F}^0(y))$, induced by $\mathscr{F}$, is a quasi-isomorphism by (we2).
\end{proof}

\begin{cor}
\label{Coro}
Given an $\mbox{A}_{\infty}$-category $\mathscr{C}$, we have that the following $\infty$-categories are weak-equivalent:
\begin{align*}
\mbox{N}_{\tiny{\mbox{A}}_{\infty}}(\mathscr{C})\simeq\mbox{N}_{\tiny{\mbox{A}}_{\infty}}(\mbox{U}(\mathscr{C}))\simeq\mbox{N}_{\tiny\mbox{dg}}(\mbox{U}(\mathscr{C})).
\end{align*}
\end{cor}

\begin{proof}
The first weak-equivalence is a consequence of Theorem \ref{M} using the fact that $\mathscr{C}\to\mbox{U}(\mathscr{C})$ is a weak-equivalence of $\mbox{A}_{\infty}$-categories, the second weak-equivalence is a straightforward consequence of {Remark} \ref{NN}.
\end{proof}

In the case of dg-categories, Lurie proved in \cite[Proposition 1.3.1.20]{Lur2} that the dg-nerve induces a right Quillen functor from the classical model structure on the category of (small) dg-categories (the first one in Example $\ref{T}$) to the Joyal model structure over sSet.\\
On the other hand, in the case of the category of $\mbox{A}_{\infty}$-categories the situation is much different.\ 
If $\mathbb{K}$ is a field, then the category of A$_{\infty}$-algebras has model structure without limits (\cite{Lef}) and Le Grignou proves in \cite{LeG} that $\mbox{N}_{\tiny\mbox{A}_{\infty}}$ preserves weak equivalences and fibrations in such a structure.\ Obviously this correspondence between equivalences and fibrations do not guarantee the existence of a right Quillen functor because of lack of limits (see \cite{COS}).\ For A$_{\infty}$-categories we have the following result:

\begin{thm}[\cite{CO}]\label{fibrazzo}
If $\mathbb{K}$ is a field then, the relative category ($\mbox{A$_{\infty}$Cat}$,$W^{\tiny\mbox{A}_{\infty}}_{\tiny\mbox{qe}}$,$F^{\tiny\mbox{A}_{\infty}}$) of $\Ain$categories, linear over $\mathbb{K}$, is a fibrant object in RelCat.\ 
An A$_{\infty}$functor $\mathscr{F}:\A\to \B$ is a \emph{fibration} if:
\begin{itemize}
\item[(F1)]  For any pair of objects $x,y\in\A$, 
\begin{align*}
\mathscr{F}_1:\A(x,y)\to \B(\F_0(x),\F_0(y))
\end{align*}
is a surjection.
\item[(F2)] 
\begin{align*}
\mbox{Ho}(\mathscr{F})\colon \mbox{Ho}(\A)\to \mbox{Ho}(\B)
\end{align*}
is a quasi-fibration\footnote{see \cite[Proposition 1.3.1.19(F)]{Lur2}}.
\end{itemize}
\end{thm}

\begin{thm}\label{ganzo}
Given a fibration $\mathscr{F}:\mathscr{A}\to\mathscr{B}$ in A$_{\infty}$-cat then $\mbox{N}_{\tiny\mbox{A}_{\infty}}(\mathscr{F})$ is a fibration. 
\end{thm}

\begin{proof}
First we prove that the functor $\mathscr{F}$ induces a quasi-fibration $\mbox{Ho}(\mbox{N}_{\tiny\mbox{A}_{\infty}}(\mathscr{F}))$ of categories.\ 
This follows by \cite[Remark 1.3.1.11 and 1.3.1.9]{Lur2} since $\mbox{Ho}(\mathscr{F})$ is a quasi-fibration of categories.\ 
By \cite[Corollary 2.4.6.5]{Lur1} it remains to prove that $\mbox{N}_{\tiny\mbox{A}_{\infty}}(\mathscr{F})$ is an inner fibration of simplicial sets.\ 
In other words, we must show that every lifting problem
\[
\xymatrix{
\Lambda^n_{j}\ar[r]^-{\phi_0}\ar[d]&\mbox{N}_{\tiny\mbox{A}_{\infty}}(\mathscr{A})\ar[d]^-{\tiny\mbox{N}_{\tiny\mbox{A}_{\infty}}(\mathscr{F})}\\
\Delta^n\ar[r]_-{\overline{\phi}}\ar@{-->}[ur]^{\phi}&\mbox{N}_{\tiny\mbox{A}_{\infty}}(\mathscr{B})
}
\]
for $0<j<n$, admits a solution.\ 
The proof is identical to the horn-filling argument in \cite[Proposition 1.3.1.20]{Lur2} or \cite[Corollary 2]{LeG}.
\end{proof}

Note that, if $\mathbb{K}$ is a commutative ring, then $\mbox{A$_{\infty}$Cat}$ is not a fibrant category, 
so to describe the homotopy category of A$_{\infty}$-categories one can use the semi-free resolutions, see \cite{Orn2}.

It is worth noting that there is a model structure on the faithful (but not full) subcategory $\mbox{A$_{\infty}$Cat}_{\tiny\mbox{strict}}\subset \mbox{A$_{\infty}$Cat}$, whose morphisms are the strict A$_{\infty}$-functors (see \cite[Theorem A]{Orn4}).\ In this model structure, the fibrations are precisely the strict A$_{\infty}$-functors satisfying (F1) and (F2) of Theorem \ref{fibrazzo}. 

\begin{rem}
\label{S}
Given a weak equivalence $F:\mbox{N}_{\tiny{\mbox{A}}_{\infty}}(\mathscr{C})\to\mbox{N}_{\tiny{\mbox{A}}_{\infty}}(\mathscr{C}')$, we have that $F$ induces an equivalence between the homotopy categories $\mbox{Ho}(\mathscr{C})\to\mbox{Ho}(\mathscr{D})$.\ 
In general it is not true that, given a weak equivalence $F:\mbox{N}_{\tiny{\mbox{A}}_{\infty}}(\mathscr{C})\to\mbox{N}_{\tiny{\mbox{A}}_{\infty}}(\mathscr{C}')$, $\mathscr{C}$ and $\mathscr{C}'$ are quasi-equivalent as $\mbox{A}_{\infty}$-categories.
\end{rem}

\section{Stable $\infty$-categories vs pretriangulated $\mbox{A}_{\infty}$-categories} 

In this section we will prove that the pretriangulated $\mbox{A}_{\infty}$-categories identified to the stable $\infty$-categories, via the $\mbox{A}_{\infty}$-nerve.

\begin{thm}[]
\label{Ttr}
Let $\mathscr{A}$ be a pretriangulated $\mbox{A}_{\infty}$-category.\ 
Then {$\mbox{N}_{\tiny{\mbox{A}}_{\infty}}(\mathscr{A})$} is a stable $\infty$-category.\ 
The functor induced between the homotopy categories is an equivalence of triangulated categories.\ 
Moreover, $\mathscr{A}$ is idempotent complete if and only if {$\mbox{N}_{\tiny{\mbox{A}}_{\infty}}(\mathscr{A})$} is an idempotent complete stable $\infty$-category.
\end{thm} 

If $\mathscr{A}$ is a pretriangulated dg-category linear over a field $\mathbb{K}$, then \cite[Theorem 3.18]{Fao}) 
proves that {$\mbox{N}_{\tiny{\mbox{dg}}}(\mathscr{A})$} is a stable $\infty$-category.\
Note that Faonte's proof of this Theorem does not work over a commutative ring.\ 
To extend this result when $\mathbb{K}$ is a commutative ring, we need to give an explicit description of the homotopy pullbacks (see example \ref{homlimcolim}) and pushouts (see example \ref{hompush}).

\begin{exmp}\label{homlimcolim}
First we have that, if $M_{\bullet}\in\mbox{Ch}^{\ge0}_{\bullet}(R)$, there exists a degrgeewise surjective quasi-isomorphism 
$\phi:\tilde{M}_{\bullet}\twoheadrightarrow M_{\bullet}$ where $\tilde{M}_{\bullet}$ is degreewise projective.\\ 
Given $X_{\bullet}\in\mbox{Ch}^{\ge0}_{\bullet}(R)$, we define the disk 
$\mathsf{D}(X_{\bullet})_{\bullet}$ as 
\begin{align*}
\mathsf{D}(X_{\bullet})_{n}=X_{n+1}\oplus X_n
\end{align*}
for $n\in\mathbb{Z}$, the differential is given by
\begin{align*}
d^{\mathsf{D}(X)}_{n}:=
\begin{pmatrix}
d^{X[1]}_{n+1} & \mbox{id}_{X_n} \\
0 & -d^X_{n} 
\end{pmatrix}
.
\end{align*}
Note that given $X_{\bullet}\in\mbox{Ch}^{\ge0}_{\bullet}(R)$ the disk $\mathsf{D}(X_{\bullet})_{\bullet}$ is acyclic.\\
Now, we suppose that $X_0$ is a fibrant object, the homotopy pullback in a model category:
\[
\xymatrix{
X_{2}\times^{h}_{X_0}X_1\ar[r]\ar[d]& X_1\ar[d]^{f_{10}}\\
X_{2}\ar[r]_{f_{20}}& X_0
}
\]
can be computed as the "usual" pullback
\[
\xymatrix{
Z_{2}\times_{X_0}Z_1\ar[r]\ar[d]& Z_1\ar@{->>}[d]^{p_{10}}\\
Z_{2}\ar@{->>}[r]_{p_{20}}& X_0
}
\]
taking the factorization of $f_{i0}:X_i\to X_0$:
\[
\xymatrix{
X_{i0}\ar@{>->}[dr]^{\sim}_-{j_{i0}}\ar[rr]^{f_{i0}}&&X_0\\
&Z_i\ar@{->>}[ur]_-{p_{i0}}&
}
\]
where $j_{i0}$ is a trivial cofibration and $p_{i0}$ is a fibration.\\
In particular, the homotopy pullback of the diagram:
\[
\xymatrix{
&0\ar[d]^-{0}\\
X_{20}\ar[r]_-{f_{20}}&X_0
}
\]
in $\mbox{Ch}^{\ge0}_{\bullet}(R)$ can be computed as the pullback of the diagram:
\[
\xymatrix{
&\mathsf{D}(\tilde{X}_0)_{\bullet}\ar[d]^-{q}\\
X_{20}\oplus \mathsf{D}(\tilde{X}_0)_{\bullet}\ar[r]_-{p}&X_0
}
\]
where 
\begin{align*}
p_n:=\begin{pmatrix}
f_{n} & 0 & \phi_n
\end{pmatrix}
\end{align*} 
and
\begin{align*}
q_n:=\begin{pmatrix}
0 & \phi_n
\end{pmatrix}
.
\end{align*}
More concretely, we can describe the homotopy pullback as:
\begin{align*}
\mbox{Ker}(\gamma):=&\mathcal{f} \big(x,(h_1,h_2),(h_3,h_4)\big)\in X_{20}\oplus \mathsf{D}(\tilde{X}_0)_{\bullet}\oplus \mathsf{D}(\tilde{X}_0)_{\bullet}\mbox{ such that }  f(x) + \phi(h_2) - \phi(h_4)=0 \mathcal{g},
\end{align*}
whose differential is given by:
\begin{align*}
d^{\tiny\mbox{Ker}(\gamma)}_{n}:=
\begin{pmatrix}
 d^{X_{20}}_{n+1} & 0 & 0 & 0 & 0\\
0 & d^{\tilde{X}_0}_{n+1} & \mbox{id}_{\tilde{X}_0} &0 &0 \\
0 & 0 & -d^{\tilde{X}_0}_{n} &0 &0 \\
0 & 0 & 0& d^{\tilde{X}_0}_{n+1} & \mbox{id}_{\tilde{X}_0} \\
0 & 0 & 0 & 0&  -d^{\tilde{X}_0}_{n} 
\end{pmatrix}
.
\end{align*}
Moreover, we have a complex
\begin{align*}
Z_{\gamma}:=&\mathcal{f} \big(x,(h_1,h_2)\big)\in X_{20}\oplus \mathsf{D}(\tilde{X}_0)_{\bullet}
\mbox{ such that }  f(x) + \phi(h_2) =0 \mathcal{g},
\end{align*}
with differential 
\begin{align*}
d^{Z_{\gamma}}_{n}:=
\begin{pmatrix}
 d^{X_{20}}_{n+1} & 0 & 0 \\
0 & d^{\tilde{X}_0}_{n+1} & \mbox{id}_{\tilde{X}_0} \\
0 & 0 & -d^{\tilde{X}_0}_{n} 
\end{pmatrix}
\end{align*}
and a morphism 
$$\psi:=Z_{\gamma}\to \mbox{Ker}(\gamma),$$
defined as 
\begin{align*}
\psi:=
\begin{pmatrix}
\mbox{id} & 0 & 0 \\
0 &\mbox{id} &0 \\
0 & 0 & \mbox{id} \\
0 &0 &0 \\
0 &0 &0 
\end{pmatrix}
.
\end{align*}
This is a quasi-isomorphism, so we can take $Z_{\gamma}=X_{20}\times^h_{X_0} 0$.
\end{exmp}

\begin{exmp}\label{hompush}
Now we give an explicit construction of the homotopy pushout in the category $\mbox{Ch}^{\ge0}_{\bullet}(R)$.\ 
We consider the diagram
\begin{align}\label{minotaz}
\xymatrix{
X_0\ar[d]_{f_{20}}\ar[r]^{f_{10}}&X_{10}\\
X_{20}&
}
\end{align}
the homotopy pushout is given by the chain complex
\begin{align*}
(X_{20}\sqcup^h_{X_0}X_{10})_n:=(X_{20})_n\oplus (X_0)_{n-1}\oplus (X_{10})_n
\end{align*}
with differential
\begin{align*}
d_n:=
\begin{pmatrix}
d^{X_20}_n & -(f_{20})_n & 0\\
0 & -d^{X_0}_{n} & 0\\
0 & (f_{10})_n & d^{X_{10}}_n
\end{pmatrix}
.
\end{align*}
For a detailed explanation see e.g. \cite[Proposition 3.29]{Hor}.
\end{exmp}

\begin{proof}[Proof of Theorem \ref{Ttr}] 
First we note that \cite[3.15 Lemma]{Fao} works over a commutative ring since \cite[Theorem 1.1. (1)]{SS} works over a commutative ring.\ It remains to prove that, given a 1-simplex $f:x\to y$ in $\mbox{N}^{\tiny\mbox{big}}_{\tiny{\mbox{dg}}}(\mathscr{A})$ (see \cite[\S 2.2]{Lur2} or \cite[Remark 2.14]{Fao}) then:
\begin{itemize}
\item[1.] The triangle in $\mbox{N}^{\tiny\mbox{big}}_{\tiny{\mbox{dg}}}(\mathscr{A})$
\begin{align}\label{fiberz}
\xymatrix{
x\ar[d]\ar[r]^f&y\ar[d]^{j}\\
0\ar[r]&\mbox{Cone($f$)}
}
\end{align}
is the cofiber of $f$ and it is cartesian.
\item[2.] The triangle in $\mbox{N}^{\tiny\mbox{big}}_{\tiny{\mbox{dg}}}(\mathscr{A})$
\begin{align}\label{cofiberz}
\xymatrix{
\mbox{Cone($f$)}[1]\ar[d]\ar[r]^-j&x\ar[d]^{f}\\
0\ar[r]&y
}
\end{align}
is the fiber of $f$ and it is cocartesian.
\end{itemize}
In other words, for every closed degree zero morphism $f$ in $\mathscr{A}$, we have to prove that (\ref{fiberz}) is a cofiber, i.e. we need to exhibit a quasi-isomorphism of chain complexes 
\begin{align}
\tau_{\ge0}(\mbox{Hom}_{\mathscr{A}}(\mbox{Cone}(f),z))^{\tiny\mbox{op}}\to \tau_{\ge0}(\mbox{Hom}_{\mathscr{A}}(y,z)^{\tiny\mbox{op}})\times^h_{\tau_{\ge0}(\tiny\mbox{Hom}_{\mathscr{A}}(x,z)^{\tiny\mbox{op}})} 0.
\end{align}
This is equivalent to prove that 
\begin{align}\label{cranchio}
\tau_{\ge0}(\mbox{Hom}_{\mathscr{A}}(\mbox{Cone}(f),z)^{\tiny\mbox{op}})
\end{align}
is quasi-isomorphic to $Z_{\gamma}$ (see Example \ref{homlimcolim}).\\ 
Since $f:x\to y$ is a closed morphism, then
$$\mbox{Cone}(f)_n:=x_{n-1}\oplus y_n,$$
whose differential is given by 
\begin{align*}
d^{\tiny\mbox{Cone}(f)}_{n}:=
\begin{pmatrix}
d &0 \\
f &-d
\end{pmatrix}
.
\end{align*}
Taking $k\ge 0$ we have:
\begin{align}\label{cranchioloz}
(\ref{cranchio}) \simeq 
\mbox{Hom}_{\mathscr{A}}^{-k}(y,z)\oplus \mbox{Hom}_{\mathscr{A}}^{-k-1}(x,z)
\end{align}
with differential 
\begin{align*}
d_{k}:=
\begin{pmatrix}
d^{-k}_{\tiny\mbox{Hom}_{\mathscr{A}}(y,z)} &0 \\
-(\mbox{-})\cdot f &-d^{-k-1}_{\tiny\mbox{Hom}_{\mathscr{A}}(x,z)}
\end{pmatrix}
.
\end{align*}
Setting $X_{20}=\mbox{Hom}_{\mathscr{A}}^{-k}(y,z)$, $X_0=\mbox{Hom}_{\mathscr{A}}^{-k-1}(x,z)$ and 
$(\mbox{-})\cdot f=f_{20}:X_{20}\to X_0$, 
we have a morphism of the chain complexes
$$\Psi:Z_{\gamma}\to X_{20}\oplus X_0[1]$$
where $\Psi$ is given by
\begin{align*}
\Psi_{n}:=
\begin{pmatrix}
\mbox{id} &0 &0 \\
0 &\phi_n &0
\end{pmatrix}
.
\end{align*}
By a direct calculation, one can show that $\Psi$ is a quasi-isomorphism of chain complexes.\\
\\
To prove that (\ref{cofiberz}) is a fiber, we exhibit a quasi-isomorphism in $\mbox{Ch}^{\ge0}_{\bullet}(R)$:
\begin{align*}
\tau_{\ge 0}(\mbox{Hom}_{\mathscr{A}}(y,z)^{\tiny\mbox{op}})&\to
\tau_{\ge0}(\mbox{Hom}_{\mathscr{A}}(x,z)^{\tiny\mbox{op}})\sqcup^h_{\tau_{\ge0}(\tiny\mbox{Hom}_{\mathscr{A}}(\tiny\mbox{Cone}(f)[1],z)^{\tiny\mbox{op}})} 0.
\end{align*}
By Example \ref{hompush} we have that $\tau_{\ge0}(\mbox{Hom}_{\mathscr{A}}(x,z)^{\tiny\mbox{op}})\sqcup^h_{\tau_{\ge0}(\tiny\mbox{Hom}_{\mathscr{A}}(\tiny\mbox{Cone}(f)[1],z)^{\tiny\mbox{op}})} 0$ is given by the chain complex:
\begin{align*}
\tau_{\ge0}(\mbox{Hom}_{\mathscr{A}}(x,z)^{\tiny\mbox{op}})\oplus {\tau_{\ge0}(\mbox{Hom}_{\mathscr{A}}(\mbox{Cone}(f)[1],z)^{\tiny\mbox{op}})}=\\
=\tau_{\ge0}(\mbox{Hom}_{\mathscr{A}}(x,z)^{\tiny\mbox{op}})\oplus {\tau_{\ge0}(\mbox{Hom}_{\mathscr{A}}(x,z)^{\tiny\mbox{op}})}\oplus {\tau_{\ge0}(\mbox{Hom}_{\mathscr{A}}(y[1],z)^{\tiny\mbox{op}})}.
\end{align*}
We have a quasi-isomorphism:
\begin{align*}
{\tau_{\ge0}(\mbox{Hom}_{\mathscr{A}}(y,z)^{\tiny\mbox{op}})}&\to \tau_{\ge0}(\mbox{Hom}_{\mathscr{A}}(x,z)^{\tiny\mbox{op}})\oplus {\tau_{\ge0}(\mbox{Hom}_{\mathscr{A}}(x,z)^{\tiny\mbox{op}})}\oplus {\tau_{\ge0}(\mbox{Hom}_{\mathscr{A}}(y[1],z)^{\tiny\mbox{op}})} \\
\alpha
&\to 
\begin{pmatrix}
\alpha\cdot f &0 &\alpha
\end{pmatrix}.
\end{align*}
It remains to prove that (\ref{fiberz}) is cartesian and (\ref{cofiberz}) is cocartesian.\ 
In other words we need to show that $x$ is homotopy equivalent to $\mbox{Cone}(j)[1]$ and $y$ is homotopy equivalent to $\mbox{Cone}(f)$.\ This is done in the proof of \cite[Theorem 3.18]{Fao}.\\
\\
Now we are ready to extend our proof to the A$_{\infty}$-categories.\ 
If $\mathscr{A}$ is a pretriangulated A$_{\infty}$-category, then ${\mbox{U}}(\mathscr{A})$ is a pretriangulated dg-category.\
By the previous step we have that the dg nerve $\mbox{N}_{\tiny\mbox{dg}}({\mbox{U}}(\mathscr{A}))$ 
is a stable $\infty$-category.\ 
By Corollary \ref{Coro}, we have that $\mbox{N}_{\tiny\mbox{dg}}({\mbox{U}}(\mathscr{A}))$ is weak-equivalent to $\mbox{N}_{\tiny\mbox{A}_{\infty}}(\mathscr{A})$ hence is a stable $\infty$-category.\ 
Moreover, by \cite[Lemma 1.2.4.6]{Lur2}, a stable $\infty$-category is idempotent complete if and only if the homotopy category is idempotent complete, so $\mathscr{A}$ is idempotent complete if and only if $\mbox{N}_{\tiny{\mbox{A}}_{\infty}}(\mathscr{A})$ is idempotent complete, and we are done.
\end{proof} 

We conclude the paper with a question.\ 
In \cite{Orn3}, we defined the (derived) tensor product $\otimes^{\mathbb{L}}$ of two A$_{\infty}$-categories, and we proved that $(\aCat,\otimes^{\mathbb{L}},R)$ forms a homotopy-coherent symmetric monoidal category.\ 
On the other hand, the category of simplicial sets has a symmetric monoidal structure.\
It remains an open question whether the A$_{\infty}$-nerve $\mbox{N}_{\tiny\mbox{A}_{\infty}}$ respects this monoidal structure.


\begin{thebibliography}{argomento}

\bibitem{BaSc} P. Balmer and M. Schlichting, \emph{Idempotent completion of triangulated categories}, J.Algebra, {\bf 236}(2): 819-834 (2001).

\bibitem{BFN} D. Ben-Zvi, J. Francis and D. Nadler, \emph{Integral transforms and Drinfeld centers in derived algebraic geometry.} J. Amer. Math. Soc., {\bf 23}: 909-966 (2010).

\bibitem{BLM} Y. Bespalov, V.V. Lyubashenko and O. Manzyuk, \emph{Pretriangulated A$_{\infty}$-categories}, Proceedings of Institute of Mathematics of NAS of Ukraine, Mathematics and its Applications 76 (2008).

\bibitem{BoKa} A. Bondal and M. Kapranov, \emph{Enhanced Triangulated Categories}, Math. USSR Sbornik, {\bf 70}: 93-107 (1991).

\bibitem{Coh} L. Cohn, \emph{Differential Graded Categories are k-linear Stable Infinity Categories}, preprint avilable at \url{https://arxiv.org/pdf/1308.2587} (2013).

\bibitem{CO} X. Chen and M. Ornaghi, \emph{A remark on fibrancy of (A$_{\infty}$Cat,W$_{A_{\infty}}^{\tiny\mbox{qe}})$}, preprint {\tt arXiv:2412.13347} (2024). 

\bibitem{COS} A. Canonaco, M. Ornaghi and P. Stellari, \emph{Localizations of the category of A$_{\infty}$-categories and internal Homs}, Doc. Math., {\bf{24}}: 2463-2492 (2019).

\bibitem{COS2} A. Canonaco, M. Ornaghi and P. Stellari, \emph{Localizations of the categories of A$_{\infty}$-categories and Internal Homs over a ring}, preprint available at \url{https://arxiv.org/pdf/2404.06610} (2024).

\bibitem{DoPu} A. Dold and D. Puppe, \emph{Homologie nicht-additiver Funktoren. Anwendungen}, Annales de l'institut Fourier, {\bf 11}:201-312 (1961). 

\bibitem{Doni} M. Doni, \emph{k-linear Morita theory}, preprint available at \url{https://arxiv.org/pdf/2406.15895} (2024).

\bibitem{Fao} G. Faonte, \emph{Simplicial Nerve of an A$_{\infty}$-categories}, Theory and Applications of Categories, {\bf 32}(2): 31-52 (2017). 

\bibitem{Joy} A. Joyal, \emph{Quasi-categories and Kan complexes}, J. Pure Appl. Algebra, {\bf 175}: 207-222, 2005.

\bibitem{Kel} B. Keller, \emph{On differential graded categories}, International Congress of Mathematicians (Madrid), Vol. II, 151-190. Eur. Math. Soc., Zurich (2006).

\bibitem{Kon} M. Kontsevich, \emph{M. Homological Algebra of Mirror simmetry}, notes on Kontsevich's talk at ICM, Zurich 1994, available at \url{https://arxiv.org/pdf/9411018v1}.

\bibitem{Hor} G. Horel, \emph{Homotopy II}, notes available at \url{ https://geoffroy.horel.org/Homotopie\%20II.pdf}.

\bibitem{Hov} M. Hovey, \emph{Model categories}, American Mathematical Society, Providence, RI (1999).

\bibitem{LeG} B. Le Grignou, \emph{From homotopy operads to infinity-operads}, J. Noncommut. Geom. {\bf 11}: 309-365 (2017).

\bibitem{Lef} K. Lef\'evre-Hasegawa, \emph{Sur les A$_{\infty}$ cat\'egories}, Ph.D. Thesis, Universit\'e Denis Diderot-Paris 7, \url{https://arxiv.org/abs/math/0310337} (2003). 

\bibitem{Lur1} J. Lurie, \emph{Higher Topos Theory}, Annals of Mathematics Studies, vol. 170, Princeton University Press, Princeton, NJ (2009).

\bibitem{Lur2} J. Lurie, \emph{Higher Algebra}, \url{https://www.math.ias.edu/~lurie/papers/HA.pdf} (2017).

\bibitem{Orn1} M. Ornaghi, \emph{Comparison results about dg categories, A$_{\infty}$-categories, stable $\infty$-categories and noncommutative motives}, Ph.D. Thesis, Milano, {\url{https://air.unimi.it/bitstream/2434/560595/2/phd_unimi_R10829.pdf}} (2019).

\bibitem{Orn2} M. Ornaghi, \emph{The Homotopy Theory of $\Ain$categories}, Int. Math. Res. Not., {\bf 2025}(11) (2025).

\bibitem{Orn3} M. Ornaghi, \emph{Tensor product of $\Ain$categories}, J. Pure Appl. Algebra, {\bf 229}(7) (2025).

\bibitem{Orn4} M. Ornaghi, \emph{A model structure on the category of A$_{\infty}$-categories with strict morphisms}, preprint available at \url{https://arxiv.org/pdf/2412.13347} (2025).

\bibitem{Qui} D. G. Quillen, \emph{Homotopical algebra}, Lecture Notes in Mathematics 43, Springer-Verlag (1967).

\bibitem{riehl} E. Riehl, \emph{Quasi-categories as $(\infty,1)$-categories},\\ 
notes available at \url{ https://emilyriehl.github.io/files/quasi-categories-as.pdf}.

\bibitem{Sei} P. Seidel, \emph{Fukaya categories and Picard-Lefschetz theory}, Zurich Lectures in Advanced Mathematics, European Mathematical Society, Zurich. MR 2441780 (2008).

\bibitem{SS} S. Schwede and B. Shipley, \emph{Equivalence of monoidal model categories}, Algebraic and Geometric Topology {\bf 3}: 287-334 (2003). 

\bibitem{Tab1} G. Tabuada, \emph{Une structure de categorie de modeles de Quillen sur la categorie des dg-catgories.}, C. R. Math. Acad. Sci. Paris {\bf 340}(1):\ 15-19 (2005).

\bibitem{Tab2} G. Tabuada, \emph{Invariants additifs de dg-categories.} Internat. Math. Res. Notices {\bf 53}:\ 3309-3339 (2005).

\end{thebibliography}
\end{document}